\documentclass{amsart}
\usepackage{amsmath,amssymb,amsthm,amsfonts}
\usepackage{hyperref}

\newtheorem{theorem}{Theorem}[section]
\newtheorem{lemma}[theorem]{Lemma}
\newtheorem{corollary}[theorem]{Corollary}
\newtheorem{proposition}[theorem]{Proposition}
\theoremstyle{definition}

\theoremstyle{remark}
\newtheorem{remark}[theorem]{Remark}

\newcommand{\dis}{\displaystyle}

\newcommand{\rmre}{{\rm Re}}

\newcommand{\R}{\mathbb{R}}

\newcommand{\T}{\mathbb{T}}

\newcommand{\semiG}{\mathbb{A}}
\newcommand{\highG}{\Theta}
\newcommand{\highB}{\Lambda}
\newcommand{\comml}{[\![}
\newcommand{\commr}{]\!]}
\newcommand{\sourceG}{g}

\newcommand{\FM}{\mathbf{M}}
\newcommand{\FP}{\mathbf{P}}
\newcommand{\FL}{\mathbf{L}}
\newcommand{\FI}{\mathbf{I}}

\newcommand{\CD}{\mathcal{D}}
\newcommand{\CE}{\mathcal{E}}
\newcommand{\CF}{\mathcal{F}}
\newcommand{\CH}{\mathcal{H}}

\newcommand{\CL}{\mathcal{L}}
\newcommand{\CN}{\mathcal{N}}
\newcommand{\CZ}{\mathcal{Z}}

\newcommand{\na}{\nabla}

\newcommand{\al}{\alpha}
\newcommand{\be}{\beta}

\newcommand{\om}{\omega}
\newcommand{\la}{\lambda}
\newcommand{\de}{\delta}

\newcommand{\pa}{\partial}

\newcommand{\eps}{\epsilon}

\newcommand{\De}{\Delta}
\newcommand{\Ga}{\Gamma}

\newcommand{\eqdef}{\overset{\mbox{\tiny{def}}}{=}}
\DeclareMathOperator{\esssup}{ess\,sup}

\begin{document}                        

\title[Decay of the Vlasov-Maxwell-Boltzmann System in $\mathbb{R}^3$]{Optimal Large-Time Behavior of the
Vlasov-Maxwell-Boltzmann System in the Whole Space}
\author{Renjun Duan}
\address{(RJD)
Department of Mathematics, The Chinese University of Hong Kong, Shatin, Hong Kong}
\email{rjduan at math.cuhk.edu.hk}
\urladdr{http://www.math.cuhk.edu.hk/~rjduan/}
\thanks{R.J.D. was partially supported by the Direct Grant 2010/2011.}

\author[R. M. Strain]{Robert M. Strain}
\address{(RMS) University of Pennsylvania, Department of Mathematics, David Rittenhouse Lab, 209 South 33rd Street, Philadelphia, PA 19104-6395, USA}
\email{strain at math.upenn.edu}
\urladdr{http://www.math.upenn.edu/~strain/}
\thanks{R.M.S. was partially supported by the NSF grant DMS-0901463, and an Alfred P. Sloan Foundation Research Fellowship.}

\begin{abstract}
In this paper we study the large-time behavior of classical
solutions to the two-species Vlasov-Maxwell-Boltzmann system in the
whole space $\R^3$.    The existence of  global in time nearby
Maxwellian solutions is known from \cite{S.VMB} in 2006.  However
the asymptotic behavior of these solutions  has been a challenging
open problem.  Building on our previous work \cite{DS-VPB} on time
decay for the simpler Vlasov-Poisson-Boltzmann system, we prove that
these solutions converge to the global Maxwellian with the optimal
decay rate of  $O(t^{-\frac{3}{2}+\frac{3}{2r}})$ in
$L^2_\xi(L^r_x)$-norm for any $2\leq r\leq \infty$ if initial
perturbation is smooth enough and decays in space-velocity fast
enough at infinity. Moreover, some explicit rates for the
electromagnetic field tending to zero are also provided.
\end{abstract}

\maketitle

\thispagestyle{empty}




\setcounter{tocdepth}{1}
\tableofcontents

\section{Introduction}

The Vlasov-Maxwell-Boltzmann system is an important physical model to describe the time evolution of dilute charged particles (e.g. electrons and ions) under the influence of their self-consistent internally generated Lorentz forces \cite{MRS}.
The existence of global in time solutions to this Cauchy problem which are perturbations of a Maxwellian equilibria is known  since 2006 in \cite{S.VMB}.
The time rate of convergence to equilibrium is an important topic in the
mathematical theory of the physical world.
It has however remained  an open problem to determine the large time behavior of these solutions.
In this work, we resolve this problem proving convergence to Maxwellian with the optimal rate of $O(t^{-\frac{3}{2}+\frac{3}{2r}})$ in $L^2_\xi(L^r_x)$-norm with $2\leq r\leq \infty$ in three spatial dimensions.

\subsection{Equation and reformulation}
 For two species of particles, the Vlasov-Maxwell-Boltzmann system is given in  its non-dimensional version as
\begin{equation}
\label{VMB.B}
\begin{split}
  \pa_t f_+ +\xi\cdot \na_x f_+ + (E+\xi \times B) \cdot \na_\xi f_+ &=
  Q(f_+,f_+)+Q(f_+,f_-),
  \\
  \pa_t f_- +\xi\cdot \na_x f_- - (E+\xi \times B) \cdot \na_\xi f_- &=
  Q(f_-,f_+)+Q(f_-,f_-).
  \end{split}
\end{equation}
It is coupled with the Maxwell system
\begin{eqnarray}
&& \pa_t E-\na_x\times B=-\int_{\R^3} \xi (f_+-f_-)d\xi,\label{VMB.M1}\\
&& \pa_t B+\na_x\times E =0,\label{VMB.M2}\\
&& \na_x\cdot E =\int_{\R^3} (f_+-f_-)d\xi,\ \ \na_x \cdot
B=0.\label{VMB.M3}
\end{eqnarray}
The initial data in this system is given as
\begin{eqnarray}\label{VMB.ID}
&\dis f_\pm(0,x,\xi)=f_{0,\pm}(x,\xi),\ \ E(0,x)=E_0(x),\ \
B(0,x)=B_0(x).
\end{eqnarray}
The initial data should satisfy the compatibility conditions
\begin{equation}
\notag
\na_x\cdot E_0=\int_{\R^3}(f_{0,+}-f_{0,-})d\xi,\ \ \na_x\cdot
B_0=0.
\end{equation}
Here the unknowns are $f_\pm=f_\pm(t,x,\xi)\geq 0$ and $E(t,x)$,
$B(t,x)$, standing for the number densities of ions
$(+)$ and electrons $(-)$ which have position
$x=(x_1,x_2,x_3)\in\R^3$ and velocity
$\xi=(\xi_1,\xi_2,\xi_3)\in\R^3$ at time $t>0$, and the
electromagnetic field, respectively. $Q$ is the bilinear Boltzmann
collision operator for the hard-sphere model defined by
\begin{eqnarray*}
&\dis Q(f,g)=\int_{\R^3\times S^{2}}(f'g_\ast'-fg_\ast)
  |(\xi-\xi_\ast)\cdot\om |d\om d\xi_\ast, \\
&\dis f=f(t,x,\xi),\ \ f'=f(t,x,\xi'), \ \ g_\ast=g(t,x,\xi_\ast), \
\ g_\ast'=g(t,x,\xi_\ast'),\\
&\dis \xi'=\xi-[(\xi-\xi_\ast)\cdot \om]\om,\ \
\xi_\ast'=\xi_\ast+[(\xi-\xi_\ast)\cdot \om]\om,\ \ \om\in S^{2}.
\end{eqnarray*}
We will study solutions to this system which are initially perturbations of the Maxwellian equilibrium states.

We write the normalized global Maxwellian as
\begin{equation}
\notag
\FM=\FM(\xi)=(2\pi)^{ -3/2}e^{-|\xi|^2/2}.
\end{equation}
Then we linearize around  {the} perturbation $u$ in the standard way
\begin{equation}
\notag
f_\pm=\FM + \FM^{1/2} u_\pm.
\end{equation}
Denote the column vector $[\cdot,\cdot]$ as follows:
$
f=[f_+,f_-],
$
$
u=[u_+,u_-].
$
Then the Cauchy
problem \eqref{VMB.B},
\eqref{VMB.M1}, \eqref{VMB.M2}, \eqref{VMB.M3},  \eqref{VMB.ID} can be
reformulated as
\begin{eqnarray}
&&\dis \pa_t u +\xi\cdot\na_x u + q (E+\xi \times B)\cdot \na_\xi u-
E\cdot \xi \FM^{1/2} q_1
\label{VMB.pB}
\\
&&\dis \hspace{6cm}=\FL u +\frac{q}{2}E\cdot \xi u +\Ga(u,u),
\notag
\\
&&\dis \pa_t E-\na_x\times B=-\int_{\R^3} \xi \FM^{1/2}(u_+-u_-)d\xi,\label{VMB.pM1}\\
&&\dis \pa_t B+\na_x\times E =0,\label{VMB.pM2}\\
&&\dis\na_x\cdot E =\int_{\R^3}\FM^{1/2} (u_+-u_-)d\xi,\ \ \na_x \cdot
B=0,
\label{VMB.pM3}
\end{eqnarray}
with initial data
\begin{eqnarray}\label{VMB.pID}
&\dis u_\pm(0,x,\xi)=u_{0,\pm}(x,\xi),\ \ E(0,x)=E_0(x),\ \
B(0,x)=B_0(x),
\end{eqnarray}
satisfying the compatibility condition
\begin{equation}
\notag
\na_x\cdot E_0=\int_{\R^3} {\FM^{1/2}}(u_{0,+}-u_{0,-})d\xi,\ \ \na_x\cdot
B_0=0.
\end{equation}
Here,  $q={\rm diag} (1,-1)$, $q_1=[1,-1]$, and the linearized
collision term $\FL u$ and the nonlinear collision $\Ga (u,u)$ are
respectively defined by
\begin{equation}
\notag
    \FL u = [\FL_+ u, \FL_- u],\ \ \Ga(u,v)=[\Ga_+(u,v),\Ga_-(u,v)],
\end{equation}
with
\begin{eqnarray*}
 \FL_\pm u &=&2\FM^{-1/2} Q(\FM^{1/2} u_\pm, \FM)+\FM^{-1/2}
    Q(\FM,\FM^{1/2}\{u_\pm+u_\mp\}),\\
\Ga_\pm (u,v)&=&\FM^{-1/2} Q(\FM^{1/2} u_\pm, \FM^{1/2}
    v_\pm)+\FM^{-1/2} Q(\FM^{1/2} u_\pm, \FM^{1/2} v_\mp).
\end{eqnarray*}
In the sequel we will study the asymptotic behavior of solutions to this system.

For the linearized collision operator $\FL$, one has the following standard facts \cite{CIP-Book}. $\FL$ can be split as $\FL u=-\nu(\xi) u+K u$, where
 the collision frequency is given by
\begin{equation}
\notag
\nu(\xi)=\int_{\R^3\times
  S^{2}}|(\xi-\xi_\ast)\cdot\om|\FM(\xi_\ast)\, d\om
  d\xi_\ast.
\end{equation}
 Notice that $\nu(\xi)\sim (1+|\xi|^2)^{1/2}$.  The null space of $\FL$ is given by
\begin{equation}\notag
    \CN={\rm span}\left\{[1,0]\FM^{1/2},\ [0,1]\FM^{1/2}, [\xi_i,\xi_i]\FM^{1/2}
 ~ (1\leq i\leq 3),
    [|\xi|^2,|\xi|^2]\FM^{1/2} \right\}.
\end{equation}
The linearized collision operator $\FL$ is non-positive and further $-\FL$ is known to be locally coercive in
the sense that there is a constant $\la_0>0$ such that \cite{CIP-Book}:
\begin{equation}\label{coerc}
    -\int_{\R^3}u\FL u\,d\xi\geq \la_0\int_{\R^3}\nu(\xi)|\{\FI-\FP\}u|^2d\xi,
\end{equation}
where, for fixed $(t,x)$, $\FP$ denotes the orthogonal projection from
$L^2_\xi$ to $\CN$.

\subsection{Notations and main results}

Throughout this paper,  $C$  denotes
some positive (generally large) constant and $\la$ denotes some positive (generally small) constant, where both $C$ and
$\la$ may take different values in different places. In addition,
$A\sim B$ means $\la A\leq B \leq \frac{1}{\la} A$ for a generic
constant $0<\la<1$. For any
integer $m\geq 0$, we use $H^m$, $H^m_x$, $H^m_\xi$ to denote the
usual Hilbert spaces $H^m(\R^n_x\times\R^n_\xi)$, $H^m(\R^n_x)$,
$H^m(\R^n_\xi)$, respectively, where $L^2$, $L^2_x$, $L^2_\xi$ are
used for the case when $m=0$. For a Banach space $X$,
$\|\cdot\|_{X}$ denotes the corresponding norm, while $\|\cdot\|$
always denotes the norm $\|\cdot\|_{L^2}$ or $\|\cdot\|_{L^2_x}$ for
simplicity.
We use $\langle\cdot,\cdot\rangle$ to denote the inner product over
the Hilbert space $L^2_\xi$, i.e.
\begin{equation*}
    \langle g,h\rangle=\int_{\R^3} g(\xi)h(\xi)d\xi,\ \ g=g(\xi),h=h(\xi)\in
    L^2_\xi,
\end{equation*}
 {and for simplicity, $\langle\cdot,\cdot\rangle$ is also used as the inner product over $L^2$ when there is no possibility of confusion.} For $r\geq 1$, we also define the standard time-space mixed Lebesgue
space $Z_r=L^2_\xi(L^r_x)=L^2(\R^3_\xi;L^r(\R^3_x))$ with the norm
\begin{equation*}
\|g\|_{Z_r}=\left(\int_{\R^3}\left(\int_{\R^3}
    |g(x,\xi)|^rdx\right)^{2/r}d\xi\right)^{1/2},\ \ g=g(x,\xi)\in Z_r.
\end{equation*}
For
multi-indices $\al=[\al_0,\al_1,\al_2,\al_3]$ and
$\be=[\be_1,\be_2,\be_3]$, we denote
\begin{equation*}
 \pa^{\al}_\be=\pa_t^{\al_0}\pa_{x_1}^{\al_1}\pa_{x_2}^{\al_2}\pa_{x_3}^{\al_3}
    \pa_{\xi_1}^{\be_1}\pa_{\xi_2}^{\be_2}\pa_{\xi_3}^{\be_3}.
\end{equation*}
The length of $\al$ is $|\al|=\al_0+\al_1+\al_2+\al_3$ and the length of
$\be$ is $|\be|=\be_1+\be_2+\be_3$. For simplicity,
we also use $\pa_j$ to denote $\pa_{x_j}$ for each $j=1,2,3$.

For an integrable function $g: \R^3\to\R$, its Fourier transform
is defined by
\begin{equation*}
  \widehat{g}(k)= \CF g(k)= \int_{\R^3} e^{-2\pi i x\cdot k} g(x)dx, \quad
  x\cdot
   k\eqdef\sum_{j=1}^3 x_jk_j,
   \quad
   k\in\R^3,
\end{equation*}
where $i =\sqrt{-1}\in \mathbb{C}$ is the imaginary unit. For two
complex vectors $a,b\in\mathbb{C}^3$, $(a\mid b)=a\cdot
\overline{b}$ denotes the dot product over the complex field, where
$\overline{b}$ is the ordinary complex conjugate of $b$. Note
further that we will use a subscript $j$ in an equation number to
denote the $j$-th equation displayed.  Precisely, for example
\eqref{VMB.B}$_2$ refers to the second equation in \eqref{VMB.B}.

Given a solution $[u(t,x,\xi),E(t,x),B(t,x)]$ to the Vlasov-Maxwell-Boltzmann system \eqref{VMB.pB}-\eqref{VMB.pID},
{an instant full energy functional is defined as a continuous function
$\CE_{N,m}(t)$ which satisfies \eqref{def.eNm}.
The instant high-order energy functional $\CE_{N,m}^{\rm h}(t)$
is defined similarly, satisfying instead \eqref{def.eNm.h}.
In particular we use}
\begin{eqnarray}
\CE_{N,m}(t)&\sim& \sum_{|\al|+|\be|\leq N}\|\nu^{\frac{m}{2}}\pa^\al_\be u(t)\|^2+\sum_{|\al|\leq N}\|\pa^\al [E(t),B(t)]\|^2,\label{def.eNm}\\
\CE_{N,m}^{\rm h}(t)&\sim& \sum_{1\leq |\al|\leq N}\|\nu^{\frac{m}{2}}\pa^\al u(t)\|^2+\sum_{|\al|+|\be|\leq N}\|\nu^{\frac{m}{2}}\pa^\al_\be  \{\FI-\FP\}u(t)\|^2
\notag
\\
&&+\sum_{1\leq |\al|\leq N}\|\pa^\al [E(t),B(t)]\|^2 {+\|E(t)\|^2},\label{def.eNm.h}
\end{eqnarray}
and we define the dissipation rate $\CD_{N,m}(t)$ as
\begin{eqnarray}
\nonumber
\CD_{N,m}(t)&\eqdef&\sum_{1\leq |\al|\leq N}\|\nu^{\frac{m+1}{2}}\pa^\al u(t)\|^2+\sum_{|\al|+|\be|\leq N}\|\nu^{\frac{m+1}{2}}\pa^\al_\be \{\FI-\FP\}u(t)\|^2\\
&&+\sum_{ {1\leq |\al|\leq N-1}}\|\pa^\al[E(t),B(t)]\|^2+\|E(t)\|^2,
\label{def.dNm}
\end{eqnarray}
for integers $N$ and $m$. For brevity, we write $\CE_N(t)=\CE_{N,0}(t)$, $\CE_N^{\rm h}(t)=\CE_{N,0}^{\rm h}(t)$ and $\CD_N(t)=\CD_{N,0}(t)$ when $m=0$. Throughout this paper we assume $N\geq 4$.

From \cite{S.VMB} and a little additional efforts about estimates of the electromagnetic field $[E(t,x),B(t,x)]$, we have the following existence result.

\begin{proposition}\label{prop.exi}
Let $[u_0,E_0,B_0]$ satisfy \eqref{VMB.pM3} and
\begin{equation*}
    f_{0,\pm}(x,\xi)=\FM+\FM^{1/2}u_{0,\pm}(x,\xi)\geq 0.
\end{equation*}
There
 {exists $\CE_N(t)$} such that if $\CE_N(0)$ is
sufficiently small, then the Cauchy problem \eqref{VMB.pB},
\eqref{VMB.pM1}, \eqref{VMB.pM2}, \eqref{VMB.pM3}, and \eqref{VMB.pID} to
the Vlasov-Maxwell-Boltzmann system admits a unique global solution
$[u(t,x,\xi),E(t,x),B(t,x)]$ with
\begin{equation*}
    f_{\pm}(t,x,\xi)=\FM+\FM^{1/2}u_{\pm}(t,x,\xi)\geq 0,
\end{equation*}
and the Lyapunov inequality
\begin{equation}\label{prop.exi.3}
    \frac{d}{dt}\CE_N(t)+\la \CD_N(t)\leq 0,
\end{equation}
for any $t\geq 0$.
\end{proposition}

Moreover, on the basis of the above existence result, one can obtain some velocity-weighted or  high-order energy inequalities stated as follows.

\begin{theorem}\label{thm.energy}
Suppose that all the assumptions of Proposition \ref{prop.exi} hold.
For any given $m=0,1,2,\ldots$,
{there exists $\CE_{N,m}(t)$}
such that if $\CE_{N,m-1}(0)$ is sufficiently small
and $\CE_{N,m}(0)$ is finite, then
\begin{eqnarray}
\frac{d}{dt}\CE_{N,m}(t)+\la \CD_{N,m}(t)\leq 0,
\label{thm.energy.1}
\end{eqnarray}
holds for any $t\geq 0$, where $\CE_{N,-1}(0)=\CE_{N}(0)$ is set and
$\la$ may depend on $m$. In addition, if $\CE_{N}(0)$ is
sufficiently small, then there is $\CE_{N}^{\rm h}(t)=\CE_{N,0}^{\rm
h}(t)$ such that
\begin{equation}\label{thm.energy.2}
\frac{d}{dt}\CE_{N}^{\rm h}(t)+\la \CD_{N}(t)\leq
C(\|
      {
     \nu^{1/2}
     }
\na_x\FP u(t)\|^2 {+\|\na_x\times B(t)\|^2}),
\end{equation}
holds for any $t\geq 0$.
\end{theorem}

From the proof later on, the above nonlinear energy estimates
together with the time-decay estimates on the linearized
non-homogeneous system indeed lead to optimal time-decay rates of
the instant energy functionals $\CE_{N,m}(t)$ $m=0,1,2,\cdots$ and
$\CE_{N}^{\rm h}(t)$  under some additional regularity and
integrability conditions on initial data. Precisely, our main
results in this paper are stated as follows. Set $\eps_{j,m}$ as
\begin{equation}\label{def.id.jm}
    \eps_{j,m}\eqdef\CE_{j,m}(0)+\|u_0\|_{Z_1}^2+\|[E_0,B_0]\|_{L^1}^2,
\end{equation}
for integers $j$ and $m$.

\begin{theorem}\label{thm.ns}
Let $[u(t,x,\xi),E(t,x),B(t,x)]$  be the
solution to   the Cauchy problem \eqref{VMB.pB},
\eqref{VMB.pM1}, \eqref{VMB.pM2}, \eqref{VMB.pM3}, and \eqref{VMB.pID} of
the Vlasov-Maxwell-Boltzmann system obtained in
Proposition \ref{prop.exi}. For any fixed $m=0,1,2,\ldots$, if $\eps_{N+2,m\vee 1}$ is sufficiently small
where $m\vee 1\eqdef\max\{m,1\}$, then
\begin{equation}\label{thm.ns.2}
    \CE_{N,m}(t)\leq C\eps_{N+2,m} (1+t)^{-\frac{3}{2}},
\end{equation}
holds for any $t\geq 0$. In addition, if $\eps_{N+5,1}$ is sufficiently small then
\begin{equation}\label{thm.ns.3}
    \CE_N^{\rm h}(t)\leq C\eps_{N+5,1} (1+t)^{-\frac{5}{2}},
\end{equation}
holds for any $t\geq 0$.
\end{theorem}

Furthermore,
we have the optimal decay rate in the $L^r$ norm for any $2 \le r \le \infty$:

\begin{corollary}\label{cor}
Let $[u(t,x,\xi),E(t,x),B(t,x)]$  be the
solution to   the Cauchy problem \eqref{VMB.pB},
\eqref{VMB.pM1}, \eqref{VMB.pM2}, \eqref{VMB.pM3}, and \eqref{VMB.pID} of
the Vlasov-Maxwell-Boltzmann system obtained in
Proposition \ref{prop.exi}. Suppose that $\eps_{9,1}$ and $\eps_{8,2}$ are sufficiently small, then
 for any $2\leq r\leq \infty$, we have the following estimates
 for any $t\geq 0$:
\begin{align}
&\|u(t)\|_{Z_r}+\|B(t)\|_{L^r_x}\leq C(1+t)^{-\frac{3}{2}+\frac{3}{2r}},\label{cor.1}\\
&\|\{\FI-\FP\}u(t)\|_{Z_r}+\|\langle[1,-1]\FM^{1/2},u(t)\rangle\|_{L^r_x}+\|E(t)\|_{L^r_x}\leq C(1+t)^{-\frac{3}{2}+\frac{1}{2r}}.\label{cor.2}
\end{align}
\end{corollary}

\begin{remark} {
From Theorem \ref{thm.ls}, which proves the time decay rates of the linearized system, the solution $[u,E,B]$
to the linearized homogeneous Vlasov-Maxwell-Boltzmann system satisfies
\begin{equation}\label{re.c1}
    \|\na_x^m u(t)\|+\|\na_x^m[E(t),B(t)]\|\leq C(1+t)^{-\frac{3}{4}-\frac{m}{2}},
\end{equation}
for any $t\geq 0$, where $m\geq 0$ is an integer, $C$ depends on the initial data, and the initial data is smooth enough  and decays fast enough in $x,\xi$ at infinity. By applying further the optimal Sobolev inequality as in \eqref{lem.semi.inf.p1} and interpolation, \eqref{re.c1} implies
\begin{equation*}
  \|u(t)\|_{Z_r}+\|[E(t),B(t)]\|_{L^r_x}\leq C(1+t)^{-\frac{3}{2}+\frac{3}{2r}}
\end{equation*}
for any $t\geq 0$, where  $2\leq r\leq \infty$. Thus, for the nonlinear system, the time decay rate in \eqref{cor.1} for $u$ and $B$ is the same as the above one in the linearized case, and the faster time decay rate in \eqref{cor.2} for $\{\FI-\FP\}u$, $\langle[1,-1]\FM^{1/2},u(t)\rangle$ and $E(t)$ is due to the special structure of the Vlasov-Maxwell-Boltzmann system. On the other hand, when the electromagnetic field disappears, the spectral analysis as in \cite{UY-AA} for the one-species Boltzmann equation in the hard sphere case shows
\begin{equation*}
    \|\na_x^m e^{t(-\xi\cdot \na_x +\FL)}u_0\|\leq C (1+t)^{-\frac{3}{2}(\frac{1}{p}-\frac{1}{2})-\frac{m}{2}}(\|u_0\|_{Z_p}+\|\na_x^m u_0\|)
\end{equation*}
for any $t\geq 0$, where $m\geq 0$ is an integer and $1\leq p\leq 2$. Thus, the time decay rate of $u$ in  \eqref{cor.1} given by
\begin{equation*}
  \|u(t)\|_{Z_r}\leq  C(1+t)^{-\frac{3}{2}+\frac{3}{2r}},\ \ 2\leq r\leq \infty,
\end{equation*}
is optimal in the sense that it is the same as in the cases of the Boltzmann equation without any force and of the linearized Vlasov-Maxwell-Boltzmann system.}
\end{remark}

Notice here that the electric field  decays faster than the magnetic field in $L^r_x$  with $2 \le r < \infty$.
Indeed recall that the pure homogeneous Maxwell's equations usually conserves the energy.  Thus the decay of these terms is truly a non-linear effect which results from the coupling with the Boltzmann equation.  {The same feature has been observed in \cite{Du-EM} for a study of the Euler-Maxwell system with relaxation, where the Green's function of the linearized system is presented in detail and thus leads to the optimal large time behavior of each component in the solution.}


\subsection{Literature, new ideas, and future possibilities}
It was pointed out by Villani in
\cite{Vi}, that there exist general structures in which the interaction
between a conservative part and a degenerate dissipative part will lead
to convergence to equilibrium; this property has been called {\it
hypocoercivity}. This paper provides a concrete
example of  hypocoercivity for the nonlinear two-species Vlasov-Maxwell-Boltzmann system in the framework of perturbations. We notice that the general theoretical framework in \cite{Vi} can not be directly applied here. On the other hand, we hope to further develop the hypocoercivity theory in the future to include some degenerately dissipative kinetic equation coupled with a class of hyperbolic systems.

There has been extensive investigations on the rate of convergence
for the nonlinear Boltzmann equation or related spatially
non-homogeneous kinetic equations with relaxations. In what follows
let us mention some of them. In the context of perturbed solutions,
the first result was given by Ukai \cite{Ukai-1974}, where the
spectral gap analysis was used to obtain the exponential rates for the
Boltzmann equation with hard potentials on torus. The results in
\cite{Ukai-1974} were improved by Ukai-Yang \cite{UY-AA} in order to
consider existence of time-periodic states in the presence of
time-periodic sources, which was later extended by
Duan-Ukai-Yang-Zhao \cite{DUYZ} to the case with time-periodic
external forcing by using the  energy-spectrum method; see also \cite{DUY} for a summary of general applications of the  energy-spectrum method when some given small external forcing occurs.
Recently, Strain-Guo \cite{SG} developed a weighted energy method to
get the exponential rate of convergence for the Boltzmann equation
and Landau equation with soft potentials on the torus.  Earlier but
along the same line of research, Strain-Guo  \cite{SG0} developed a
general theory of polynomial decay rates up to any order in a unified
framework and applied it to four kinetic equations, the
Vlasov-Maxwell-Boltzmann System, the relativistic Landau-Maxwell
System, the Boltzmann equation with cutoff soft-potentials and the
Landau equation all on the torus.

Another  tool is entropy method which has  general
applications in the existence theory for nonlinear equations. By
using this method as well as the elaborate analysis of  functional
inequalities, time-derivative estimates  and interpolation,
Desvillettes-Villani \cite{DV} obtained first the almost exponential
rate of convergence of solutions to the Boltzmann equation on torus
with soft potentials for large initial data under the additional
regularity conditions that all the moments of $f$ are uniformly
bounded in time and $f$ is bounded in all Sobolev spaces uniformly
in time. See Villani \cite{Vi} for extension and simplification of
results in \cite{DV} still conditionally to smoothness bounds by
further designing a new auxiliary functional. Notice that \cite{SG0}
provided a very simple proof of \cite{DV} for the unconditional
perturbative regime. Recently, by finding some proper Lyapunov
functional defined over the Hilbert space, Mouhot-Neumann \cite{MN}
obtained the exponential rates of convergence for some kinetic
models with general structures in the case of torus; see also
\cite{Vi} for the  general study. An extension of \cite{MN} to
models with a single mass conservation law in the presence of some confining potential forces
was given by
Dolbeault-Mouhot-Schmeiser \cite{DMS}. Recently, Duan \cite{D-Hypo} developed a general method to deal with a class of linear degenerately dissipative kinetic equations with or without confining forces even when several physical conservation laws are present, specifically the linear Boltzmann equation with a parabolic confining force included. We remark that some ideas in \cite{D-Hypo} for finding the dissipation of all the macroscopic components of the kinetic equation  will be essentially used in the proof of our main results of this paper.

Besides those methods mentioned above for the study of rates of
convergence, the method of Green's functions was also founded by
Liu-Yu \cite{Liu-Yu-1DG} to expose the pointwise large-time behavior
of solutions to the Boltzmann equation in the full space $\R^3$.

In addition, concerning the Vlasov-Poisson-Boltzmann system, Glassey-Strauss \cite{GS-DCDS} studied the
essential spectra of the solution operator and they \cite{GS-TTSP} also used the method of thirteen moments by Kawashima \cite{Ka-BE13} to obtain the exponential time-decay rates of solutions on torus. A slow time-decay rate for the solution in the whole space was obtained by  {Yang-Zhao} \cite{YZ-CMP} on the basis of the pure energy method and a time differential inequality. Notice the series of works \cite{Duan,D-09PD,DY-09VPB}  investigating solution  spaces without initial layer.
Then Duan-Strain \cite{DS-VPB} obtained the optimal time-decay rates of solutions to the one-species Vlasov-Poisson-Boltzmann system in $\R^3$. The key of the method to study
hypocoercivity provided by  \cite{DS-VPB} is to carefully capture the
full dissipation of the perturbed macroscopic system of equations
with the hyperbolic-parabolic structure, which is in the same spirit
of the Kawashima's work \cite{Ka}. For the linearized time-decay  analysis, instead of using the compensation function as in \cite{Ka,Ka-BE13}, the main idea of \cite{DS-VPB,D-Hypo} is to design some interactive functionals in order to take care of the dissipation of the degenerate part in the solution.

Here, we mention that if there is no collisional effects as in
 the Vlasov-Poisson
system, then the so-called Landau damping comes out. This was
recently studied by Mouhot-Villani \cite{MV} on torus, where it was
shown that even though in the absence of kinetic relaxation, in the
analytic regime the solution still converges weakly  to
 large-time states determined by the initial data and the
nonlinear system itself.

For the more intricate Vlasov-Maxwell-Boltzmann system, which includes the hyperbolic coupling, as mentioned before, rapid polynomial time decay up to any order on the torus (for the  solutions from \cite{GuoVMB}) was shown  in Strain-Guo \cite{SG0}. Different from the case of torus, solutions in the whole space are dispersive and hence only the slower polynomial time decay is expected as in \cite{DS-VPB}. Moreover, because of the weaker dissipative property of the Maxwell equations than in the case of the Poisson equation \cite{DS-VPB}, it is difficult to study the large-time behavior of solutions to the  Vlasov-Maxwell-Boltzmann system in the whole space.

Finally, we also mention some of  {the} results on the existence theory of
the simpler Vlasov-Poisson-Boltzmann system and related kinetic equations: global existence of
renormalized weak solutions with large initial data was studied in
\cite{DL.BE,DL-CPAM,Mischler}, global existence of classical
solutions near Maxwellians can be found in \cite{Liu-Yu-Shock,Liu-Yang-Yu,YYZ-ARMA},
\cite{Guo2,Guo-IUMJ,S.VMB} and \cite{Duan,DY-09VPB}, and global
existence of  solutions near vacuum was shown in \cite{Guo3,DYZ-DCDS}.
We anticipate that the time derivatives may be removed from the existence theory \cite{S.VMB}
and our main results herein using methods from \cite{J.DIFF,Duan}.

\medskip

To obtain the results from this paper in Theorem \ref{thm.energy}, Theorem \ref{thm.ns}, and Corollary \ref{cor} the starting point is our first paper \cite{DS-VPB} and the existence theorem from \cite{S.VMB}.  There are additionally several new ingredients which we expect can be useful in a number of other contexts, some of which we list after explaining the key new elements in our proof just now.  First of all, from the hyperbolic nature of Maxwell equations it seems to be necessary to lose a derivative in estimates of the dissipation rate for the electro-magnetic field.  This makes it very difficult to prove time decay without paying the price of extra derivatives.  This problem is also present in the case of the torus \cite{SG0}, but it is much more problematic in the whole space because the dissipation is necessarily
dramatically weaker \cite{S.VMB}.    To surmount this difficulty in the whole space, we first spend some time developing a new pointwise time-frequency Lyapunov functional which measures directly the weak dissipation (and spatial derivative loss) of the full Vlasov-Maxwell-Boltzmann system; this is given in Theorem \ref{thm.tfli}.  Here new estimates are developed to measure the macroscopic dissipation and the very weak electro-magnetic dissipation in the whole space by introducing two interactive time-frequency functionals,
$$
\CE_{\rm int}^{(1)}(t,k), \ \ \CE_{\rm int}^{(2)}(t,k),
$$
which are defined by \eqref{def.int1} and \eqref{def.int2}, respectively. After this we can  make use of the usual frequency splitting method in Section \ref{sec.tf} in order to take advantage of the weak dissipation and prove the linear decay which is stated in Theorem \ref{thm.ls}.
We furthermore discuss a new time-frequency (time-derivative) splitting method in Remark \ref{time.f.m} which is related to the general derivative loss  {phenomena}.
For the low frequency part, the solution decays like the heat kernel, while for the high  frequency part, it decays in time with the algebraic rates of any order as long as initial data is regular enough in space variable. We are optimistic that this approach can be useful in other hyperbolic or mixed hyperbolic-parabolic systems where derivative loss is present.  {For the study of some macroscopic systems with this kind of regularity-loss property, we here mention the recent work \cite{IHK,IK,UWK} by Kawashima and his collaborators.} Another new point is to introduce the time-weighted method to handle the time-decay rates of the full instant energy functional, which seems necessary because of the regularity-loss property of the electromagnetic field.
Furthermore, to obtain the optimal time decay rates in $L^r_x$ with $2\le r \le \infty$ for the separate components of the solution as stated in Corollary \ref{cor}  we use  an optimized Sobolev inequality as in \eqref{lem.semi.inf.p1}.

To finish we mention some expected future applications of these methods.
As a result of scaling and the Lorentz invariant Maxwell equations, a very physically relevant collision kernel to apply when the kinetic equation is coupled with its  internally generated electromagnetic forces is the relativistic Landau-Maxwell system, for which the existence theory is known from \cite{SG1}.  We expect that with a suitable relativistic modification of the moment equations in Section \ref{sec.me},
then the methods in this paper and \cite{DS-VPB} can be used to observe that
 the all of results herein with the same conclusions can be obtained for the relativistic Landau-Maxwell system in the whole space.  The details of this approach  would be quite complicated and interesting to carry out.

Additionally we expect that the methods developed in this paper and \cite{DS-VPB} can be useful in several other physical contexts.  In particular one should be able to combine the methods from \cite{SG0} with our methods here to obtain optimal convergence rates for the relativistic Boltzmann equation with soft potentials \cite{S.rB.soft} in the whole space.  More recently, a global existence theory has been developed for perturbative solutions to the Boltzmann equation without the angular cut-off assumption for the full range of inverse power-law potentials on the torus \cite{GS.B.NCA,GS.B.NC1}.   The methods used there yield exponential decay for hard and moderately soft potentials ($\gamma + 2s \ge 0$), as well as ``almost exponential decay'' for the softest potentials.  The methods developed in this paper and \cite{SG0,DS-VPB} have also been quite useful in obtaining the optimal decay rates in the whole space for the these solutions \cite{GS.B.NCA,GS.B.NC1} for all of the hard and soft-potentials
 {\cite{S.B.NC0}.}   {Finally, we would like to mention that in the direction of the non cut-off Boltzmann equation, a study about the qualitative properties of classical solutions, precisely, the full regularization in all variables, uniqueness, non-negativity and the convergence rates to the equilibrium, has been also provided by
Alexandre, Morimoto, Ukai, Xu, and Yang
 in \cite{AMUXY1} and \cite{AMUXY2} (and the references therein), where the optimal time decay rates for the hard potentials are obtained by using the compensation function method developed in \cite{Ka-BE13}.}

\subsection{Organization of the paper}
In Section \ref{sec.me}, we define the macroscopic projector and derive some moment equations up to third-order. In Section \ref{sec.decayl}, we prove Theorem \ref{thm.ls} about the time-decay property of solutions to the linearized Vlasov-Maxwell-Boltzmann system with microscopic source terms.  This is based on Theorem \ref{thm.tfli} for the construction of a time-frequency Lyapunov functional resulting from estimates on the microscopic dissipation, macroscopic dissipation and electromagnetic dissipation  proved in Section \ref{sec.sub.tfli}. Then, the proof of Theorem \ref{thm.ls} is given in Section \ref{sec.tf}. In Section \ref{sec.energy}, we study some nonlinear energy estimates including the time evolution of the equivalent velocity-weighted instant total energy and the high-order energy to finish the proof of Proposition \ref{prop.exi} and Theorem \ref{thm.energy}.
Then in Section \ref{sec.decayNL}, we use the decay of the linearized solution with nonhomogeneous source terms together with those nonlinear energy estimates to bootstrap and obtain the nonlinear decay. In Section \ref{sec.decayNL.1} and Section \ref{sec.decayNL.2}, we prove Theorem \ref{thm.ns}. Finally in Section \ref{sec.decayNL.3}, we prove Corollary \ref{cor}.

\section{Moment equations}\label{sec.me}

In this section, we begin with the representation of the macroscopic
projector $\FP$ and then derive from the perturbed system some
macroscopic balance laws and high-order moment equations which are
systems
 {of first order} hyperbolic equations of the macroscopic
coefficient functions coupled with the high-order moment functions.

Given any $u(t,x,\xi)$, one can write $\FP$ in \eqref{coerc} as
\begin{multline}
 \FP u=a_+(t,x)[1,0]\FM^{1/2}+a_-(t,x)[0,1]\FM^{1/2}
    \\
    +\sum_{i=1}^3b_i(t,x)[1,1]\xi_i\FM^{1/2}+c(t,x)[1,1](|\xi|^2-3)\FM^{1/2},\label{form.p}
\end{multline}
since $\FP$ is a projection from $L^2_\xi\times L^2_\xi$ to $\CN$, where the coefficient functions $a_\pm(t,x)$, $b(t,x)\equiv [b_1(t,x),b_2(t,x),b_3(t,x)]$ and $c(t,x)$ depend on $u(t,x,\xi)$. The expression \eqref{form.p} can then be rewritten as
$\FP u=[\FP_+ u, \FP_- u]$ with
\begin{equation}  \notag
    \FP_\pm u= \{a_\pm(t,x)+b(t,x)\cdot \xi+c(t,x)(|\xi|^2-3)\}\FM^{1/2}.
\end{equation}
Since the projection $\FP$ is orthogonal we have
\begin{equation*}
     \int_{\R^3} \psi(\xi)\cdot \{\FI-\FP\} u ~ d\xi=0,\ \ \forall\,\psi = [\psi_+, \psi_-]\in
    \CN,
\end{equation*}
which together with the
form \eqref{form.p} of $\FP$ imply
\begin{equation}
\notag
\begin{split}
  \dis   & a_\pm= \langle \FM^{1/2}, u_\pm\rangle= \langle \FM^{1/2}, \FP_\pm u\rangle,
  \\
  \dis  & b_i=\frac{1}{2}\langle \xi_i \FM^{1/2}, u_++u_-\rangle
=\langle \xi_i \FM^{1/2},\FP_\pm u\rangle,
\\
  \dis & c= \frac{1}{12}\langle (|\xi|^2-3) \FM^{1/2}, u_+ + u_-\rangle
= \frac{1}{6}\langle (|\xi|^2-3) \FM^{1/2}, \FP_\pm u\rangle.
    \end{split}
\end{equation}
In the rest of this section we will derive the equations for these macroscopic variables and also the high-order moments as follows.

First consider the following linearized system with a non-homogeneous source $\sourceG=[\sourceG_+(t,x,\xi),\sourceG_-(t,x,\xi)]$:
\begin{equation}\label{lsns}
    \left\{\begin{array}{l}
  \dis     \pa_t u_\pm+\xi\cdot\na_x u_\pm \mp E\cdot \xi \FM^{1/2}=\FL_\pm u +\sourceG_\pm,\\
 \dis \pa_t E-\na_x\times B=-\int_{\R^3} \xi \FM^{1/2}(u_+-u_-)d\xi,\\
\dis \pa_t B+\na_x\times E =0,\\
\dis\na_x\cdot E =\int_{\R^3} \FM^{1/2} (u_+-u_-)d\xi,\ \ \na_x \cdot
B=0.
    \end{array}\right.
\end{equation}
Taking velocity integrations of \eqref{lsns}$_1$ with respect to the velocity moments
\begin{equation}
\notag
    \FM^{1/2},\ \  \xi_i  \FM^{1/2}, i=1,2,3,\ \  \frac{1}{6}(|\xi|^2-3)\FM^{1/2},
\end{equation}
one has
\begin{eqnarray}
&&\quad \pa_t a_\pm +\na_x\cdot b +\na_x \cdot \langle \xi
\FM^{1/2},\{\FI_\pm-\FP_\pm\} u\rangle=\langle\FM^{1/2},\sourceG_\pm\rangle,\label{m0}\\
&& \pa_t [b_i+  \langle \xi_i \FM^{1/2},\{\FI_\pm-\FP_\pm\} u\rangle
]+\pa_i (a_\pm+2c)\mp E_i
\notag
\\
&&\qquad +\na_x\cdot \langle \xi\xi_i
\FM^{1/2},\{\FI_\pm-\FP_\pm\} u\rangle=\langle\xi_i  \FM^{1/2},\sourceG_\pm {+\FL_\pm u}\rangle,\label{m1}\\
&&\pa_t \left[c+ \frac{1}{6}\langle (|\xi|^2-3)\FM^{1/2},\{\FI_\pm-\FP_\pm\}
u\rangle
\right]+  \frac{1}{3} \na_x\cdot b
\notag
\\
&&\qquad +
\frac{1}{6}\na_x\cdot \langle
 (|\xi|^2-3)\xi\FM^{1/2},\{\FI_\pm-\FP_\pm\} u\rangle
 =\frac{1}{6}\langle (|\xi|^2-3)\FM^{1/2},\sourceG_\pm  {+ \FL_\pm u}\rangle,
 \label{m2}
\end{eqnarray}
where we have set $\FI=[\FI_+,\FI_-]$ with $\FI_\pm u=u_\pm$. Define the high-order moment
 functions $\highG(u_\pm)=(\highG_{ij}(u_\pm))_{3\times 3}$ and
$\highB(u_\pm)=(\highB_1(u_\pm),\highB_2(u_\pm),\highB_3(u_\pm))$ by
\begin{equation}
  \highG_{ij}(u_\pm) = \langle(\xi_i\xi_j-1)\FM^{1/2}, u_\pm\rangle,\ \
  \highB_i(u_\pm)=\frac{1}{10}\langle(|\xi|^2-5)\xi_i\FM^{1/2},
  u_\pm\rangle.\label{def.gala}
\end{equation}
Further taking velocity integrations of \eqref{lsns}$_1$ with respect to the above high-order
moments one has
\begin{eqnarray}
&&\pa_t [\highG_{ii}( \{\FI_\pm-\FP_\pm\}u)+2c]+2\pa_i b_i
=\highG_{ii}(l_{\pm}+\sourceG_\pm),\label{m2ii}\\
&&\pa_t \highG_{ij}( \{\FI_\pm-\FP_\pm\}u) +\pa_j b_i+\pa_i b_j +\na_x\cdot
\langle \xi\FM^{1/2},\{\FI_\pm-\FP_\pm\} u\rangle
\notag
\\
&&\hspace{4cm}=\highG_{ij}(l_{\pm}+\sourceG_\pm)+\langle\FM^{1/2},\sourceG_\pm\rangle,\ \ i\neq j,\label{m2ij}\\
&& \pa_t \highB_i( \{\FI_\pm-\FP_\pm\}u)+\pa_i c=\highB_i(l_{\pm}+\sourceG_\pm),\label{m3}
\end{eqnarray}
where
\begin{equation}\label{def.l}
l_\pm =-\xi \cdot \na_x \{\FI_\pm-\FP_\pm\}u+\FL_\pm u.
\end{equation}
Here we  used the moment values of the normalized
global Maxwellian $\FM$:
\begin{eqnarray*}
&&\langle 1, \FM\rangle=1, \ \
\langle |\xi_j|^2, \FM\rangle=1,\ \ \langle |\xi|^2, \FM\rangle=3,\\
&&\langle |\xi_j|^2|\xi_m|^2, \FM\rangle=1, \ \ j\neq m,\\
&&\langle |\xi_j|^4, \FM\rangle=3,\ \ \langle |\xi|^2|\xi_j|^2,
\FM\rangle=5.
\end{eqnarray*}
Additionally to derive \eqref{m2ij} we have used \eqref{m0}.

In particular, for the nonlinear system  {\eqref{VMB.pB}-\eqref{VMB.pM3}}, the non-homogeneous source $\sourceG=[\sourceG_+(t,x,\xi),\sourceG_-(t,x,\xi)]$ takes the form of
\begin{equation}\label{def.g.non}
    \sourceG_\pm=\pm\frac{1}{2}E\cdot \xi u_\pm \mp (E+\xi\times B)\cdot \na_\xi u_\pm +\Ga_\pm(u,u).
\end{equation}
Then, it is straightforward to compute from integration by parts that
\begin{eqnarray*}
 \langle \FM^{1/2},\sourceG_\pm\rangle &=& 0,\\
   \langle \xi\FM^{1/2},\sourceG_\pm\rangle &=&\pm Ea_\pm \pm b\times B  {\pm} \langle\xi\FM^{1/2}, \{\FI_\pm-\FP_\pm\} u\rangle \times B\\
   && {+\langle \xi\FM^{1/2},\Ga_\pm (u,u)\rangle},\\
    \frac{1}{6}\langle (|\xi|^2-3)\FM^{1/2},\sourceG_\pm\rangle &=&\pm \frac{1}{3} b\cdot E \pm \frac{1}{3}
     \langle\xi\FM^{1/2}, \{\FI_\pm-\FP_\pm\} u\rangle\cdot E\\
     && {+\langle \frac{1}{6}(|\xi|^2-3)\FM^{1/2},\Ga_\pm (u,u)\rangle}.
\end{eqnarray*}
Thus, the balance laws \eqref{m0}-\eqref{m2} for the general case can be  re-written as
\begin{eqnarray}
&&\pa_t a_\pm +\na_x\cdot b +\na_x \cdot \langle \xi
\FM^{1/2},\{\FI_\pm-\FP_\pm\} u\rangle=0,\label{m0.non}\\
&& \pa_t [b_i+  \langle \xi_i \FM^{1/2},\{\FI_\pm-\FP_\pm\} u\rangle
]+\pa_i (a_\pm+2c)\mp E_i
\notag
\\
\notag
&&\qquad+\na_x\cdot \langle \xi\xi_i
\FM^{1/2},\{\FI_\pm-\FP_\pm\} u\rangle =\pm E_ia_\pm \pm [b\times B]_i\\
&&\qquad\qquad {\pm}[\langle\xi\FM^{1/2}, \{\FI_\pm-\FP_\pm\} u\rangle \times B]_i
 {+\langle \xi_i\FM^{1/2},\FL_\pm u+\Ga_\pm (u,u)\rangle},\label{m1.non}\\
&&\pa_t \left[c+ \frac{1}{6}\langle (|\xi|^2-3)\FM^{1/2},\{\FI_\pm-\FP_\pm\}
u\rangle
\right]+  \frac{1}{3} \na_x\cdot b
\notag
\\
\notag
&&\qquad+
\frac{1}{6}\na_x\cdot \langle
 (|\xi|^2-3)\xi\FM^{1/2},\{\FI_\pm-\FP_\pm\} u\rangle\\
 \notag
 &&\qquad\qquad=\pm \frac{1}{3} b\cdot E \pm \frac{1}{3}
     \langle\xi\FM^{1/2}, \{\FI_\pm-\FP_\pm\} u\rangle\cdot E\\
 &&\qquad\qquad\qquad     {+\langle \frac{1}{6}(|\xi|^2-3)\FM^{1/2},\FL_\pm u+\Ga_\pm (u,u)\rangle}.
 \label{m2.non}
\end{eqnarray}
Furthermore, the system of the high-order moments \eqref{m2ii}-\eqref{m3} becomes
\begin{eqnarray}
&&\pa_t [\highG_{ii}( \{\FI_\pm-\FP_\pm\}u)+2c]+2\pa_i b_i
=\highG_{ii}(l_{\pm}+\sourceG_\pm),\label{m2ii.non}\\
&&\pa_t \highG_{ij}( \{\FI_\pm-\FP_\pm\}u) +\pa_j b_i+\pa_i b_j +\na_x\cdot
\langle \xi\FM^{1/2},\{\FI_\pm-\FP_\pm\} u\rangle
\notag
\\
&&\hspace{4cm}=\highG_{ij}(l_{\pm}+\sourceG_\pm),\ \ i\neq j,\label{m2ij.non}\\
&& \pa_t \highB_i( \{\FI_\pm-\FP_\pm\}u)+\pa_i c=\highB_i(l_{\pm}+\sourceG_\pm).\label{m3.non}
\end{eqnarray}
The difference between these and \eqref{m2ii}-\eqref{m3} is in \eqref{m2ij.non} since
$
\langle \FM^{1/2},\sourceG_\pm\rangle = 0.
$

We conclude this section with some remarks. The derivation of the
system \eqref{m2ii}-\eqref{m3} or \eqref{m2ii.non}-\eqref{m3.non}
was initiated by \cite{Guo2,Guo-IUMJ}, developed for VMB in \cite{S.VMB}, and refined in \cite{D-09PD}
by firstly introducing the high-order moment function $\highG$ and
$\highB$. These systems play an essential role in the Fourier
analysis of the linearized system with general microscopic sources
as in the case of the Vlasov-Poisson-Boltzmann system \cite{DS-VPB}
and even in the study of the hypocoercivity of some linear
degenerately dissipative kinetic equations  \cite{D-Hypo}. In fact,
they are also inspired by the earlier investigation of solution
spaces without any time derivatives for the well-posedness of the
Cauchy problem on the pure Boltzmann equation \cite{Duan} and the
Vlasov-Poisson-Boltzmann system \cite{DY-09VPB}.



\section{The linearized system with micro sources}
\label{sec.decayl}


In this section, we are concerned with time-decay properties of
solutions to the Cauchy problem on the linearized
Vlasov-Maxwell-Boltzmann system with microscopic sources.
Specifically, we state the main result in the first subsection, and
derive a Lyapunov-type inequality for pointwise time-frequency
variables in the second subsection. Temporal decay rates of the solution and
its derivatives in $L^2$-norms are obtained under some regularity
and integrability conditions on initial data and the source terms in the last
subsection.

\subsection{Time-decay properties of solutions}

Consider the Cauchy problem on the linearized system with a
microscopic source $\sourceG=\sourceG(t,x,\xi) = [\sourceG_+,
\sourceG_-]$:
\begin{equation}\label{ls}
    \left\{\begin{array}{l}
  \dis     \pa_t u+\xi\cdot\na_x u - E\cdot \xi \FM^{1/2}q_1=\FL u +\sourceG,\\
 \dis \pa_t E-\na_x\times B=-\langle \xi \FM^{1/2}, u_+-u_-\rangle,\\
\dis \pa_t B+\na_x\times E =0,\\
\dis\na_x\cdot E =\langle\FM^{1/2}, u_+-u_-\rangle,\ \ \na_x \cdot
B=0,\\
\dis [u,E,B]|_{t=0}=[u_0,E_0,B_0],
    \end{array}\right.
\end{equation}
where $\sourceG=\{\FI-\FP\} \sourceG$ and $[u_0,E_0,B_0]$ satisfies the
compatibility condition
\begin{equation}\label{comp.con}
\na_x\cdot E_0=\int_{\R^3}\FM^{1/2}(u_{0,+}-u_{0,-})d\xi,\ \ \na_x\cdot
B_0=0.
\end{equation}
For simplicity, we write
\begin{equation}
\notag
    U=[u,E,B],\ \ U_0=[u_0,E_0,B_0].
\end{equation}
Formally, the solution to the Cauchy problem \eqref{ls} is denoted by
\begin{eqnarray}
&& U(t)=U^{I}(t)+U^{II}(t),\label{def.lu}\\
&& U^{I}(t)=\semiG(t)U_0,\ \ U^{I}=[u^I,E^I,B^I], \label{def.lu1}\\
&& U^{II}(t)=\int_0^t\semiG(t-s)[\sourceG(s),0,0]ds,\ \ U^{II}=[u^{II},E^{II},B^{II}],\label{def.lu2}
\end{eqnarray}
where $\semiG(t)$ is the linear solution operator for the Cauchy problem
on the linearized homogeneous system corresponding to  \eqref{ls} with $\sourceG=0$.
Notice that $U^{II}(t)$ is well-defined because $[\sourceG(s),0,0]$ for
any $0\leq s\leq t$ satisfies the compatibility condition \eqref{comp.con}
due to the fact that $\FP \sourceG(s)=0$ and hence
\begin{equation}  \notag
    \int_{\R^3}\FM^{1/2}[\sourceG_+(s)-\sourceG_-(s)]d\xi=0.
\end{equation}
For brevity, we introduce the norms
$\|\cdot\|_{\CH^m}$, $\|\cdot\|_{\CZ_r}$ with $m\geq 0$ and $r\geq
1$ given by
\begin{equation}\label{brief.norm}
    \|U\|_{\CH^m}^2=\|u\|_{L^2_\xi(H^m_x)}^2+\|[E,B]\|_{H^m_x}^2,\ \  \|U\|_{\CZ_r}=\|u\|_{Z_r}+\|[E,B]\|_{L^r_x},
\end{equation}
for $U=[u,E,B]$, and we set $\CL^2=\CH^0$ as usual. The main result
of this section is stated as follows.


\begin{theorem}\label{thm.ls}
Let $1\leq r\leq 2$,  $\ell\geq 0$, and let $m\geq 0$ be an integer.
Suppose that \eqref{comp.con} and $\FP \sourceG=0$ hold. Assume that
$U$ is defined in \eqref{def.lu}, \eqref{def.lu1} and
\eqref{def.lu2} as the solution to the Cauchy problem \eqref{ls}.
Then, the first part $U^{I}$ corresponding to the solution of the
linearized homogeneous system satisfies
\begin{equation}
 \|\na_x^m U^I(t)\|_{\CL^2}
\leq C(1+t)^{-\frac{3}{2}(\frac{1}{r}-\frac{1}{2})-\frac{m}{2}}
\|U_0\|_{\CZ_r}
 +C(1+t)^{-\frac{\ell}{2}} \|\na_x^{m+\ell}U_0\|_{\CL^2},
 \label{thm.ls.1}
\end{equation}
for any $t\geq 0$, and the second part $U^{II}$ corresponding to  the solution of the linearized nonhomogeneous system  with vanishing initial data satisfies
\begin{multline}
 \|\na_x^m U^{II}(t)\|_{\CL^2}^2
\leq C\int_0^t (1+t-s)^{-3(\frac{1}{r}-\frac{1}{2})-m}\|\nu^{-1/2}\sourceG(s)\|_{Z_r}^2ds\\
+C \int_0^t (1+t-s)^{-\ell} {\|\nu^{-1/2}\na_x^{m+\ell}\sourceG(s)\|^2}\,ds,
\label{thm.ls.2}
\end{multline}
for any $t\geq 0$. Here, if $\ell$ is not integer, $\na_x^\ell$ is regarded as the fractional spatial derivative in terms of the Fourier transform.
\end{theorem}

The above theorem shows that solutions to the linearized homogeneous
Vlasov-Maxwell-Boltzmann system decays in time with explicit rates whenever  {the} initial data has enough integrability and regularity. In
fact, in \eqref{thm.ls.1}, the first term on the r.h.s. is generated
by the lower-frequency part of solutions for which time rates are
consistent with those of the classical heat equation, while the
second term results from the high-frequency part of solutions for
which higher regularity of initial data implies faster time decay rates.
The Maxwell equations of the electromagnetic field are essentially
responsible for temporal  {rates with extra} initial regularity. This kind
of  phenomenon does not happen to the case of the
Vlasov-Poisson-Boltzmann system \cite{DS-VPB}, where the second term
on the r.h.s. of  \eqref{thm.ls.1} disappears. On the other hand, this behavior is consistent with the Vlasov-Maxwell-Boltzmann system on the torus
\cite{SG0}. In fact, from the proof of Theorem \ref{thm.ls} later, it is easy to see that the first term
on the r.h.s. of  \eqref{thm.ls.1} should disappear if the spatial domain is  {the} torus, because this term results from the estimates on the low frequency part of solutions and it is also noticed that the frequency variable $k$ takes the discrete values for the torus case. Thus, our proof for Theorem \ref{thm.ls} can provide a better understanding of the time-decay property of the Vlasov-Maxwell-Boltzmann system at the linearized level for both the whole space $\R^3$ and the torus $\T^3$ cases.

The  strategy  for the proof of Theorem  \ref{thm.ls} is   to
construct a time-frequency Lyapunov functional $\CE(t,k)$
corresponding to the Fourier transform of the system \eqref{ls} such
that the functional is not only equivalent with  {the} energy-type norm
$$
\|\hat{u}(t,k)\|_{L^2_\xi}^2+|\hat{E}(t,k)|^2+|\hat{B}(t,k)|^2,
$$
but also its dissipation rate can be characterized by the functional
itself; see Theorem \ref{thm.tfli}. Indeed, the dissipation rate is proportional to the above energy-type norm with coefficient
$$
p(k)=\frac{\la |k|^2}{(1+|k|^2)^2}.
$$
Once the Lyapunov-type inequality of $\CE(t,k)$ is obtained for pointwise  time-frequency
variables, time-decay of solutions in the physical space follows
from the analysis of the frequency integration over  low and
high frequency domains for which $p(k)$ has the different pointwise behaviors.

\subsection{A time-frequency Lyapunov inequality}\label{sec.sub.tfli}

In this subsection, we shall construct the desired time-frequency
Lyapunov functional as motioned before. The proof will be carried
out along the similar line as in \cite{DS-VPB}, but additional
efforts need to be made to take care of the weak dissipation of the
electromagnetic field.

\subsubsection{Estimate on the micro dissipation}

The first step for the construction of the time-frequency Lyapunov
functional is to obtain the micro dissipation on the basis of the
coercivity property \eqref{coerc} of $-\FL$. For simplicity, here
and in the sequel, write $ G=\langle \xi \FM^{1/2}, u_+-u_-\rangle$.
Then,
\begin{equation}
\notag
   G=\langle [\xi,-\xi]\FM^{1/2}, \{\FI-\FP\} u\rangle=\langle \xi\FM^{1/2}, \{\FI-\FP\} u\cdot q_1\rangle.
\end{equation}
Thus, \eqref{ls}$_1$-\eqref{ls}$_4$ also reads
\begin{equation}
\notag
    \left\{\begin{array}{l}
  \dis     \pa_t u+\xi\cdot\na_x u - E\cdot \xi \FM^{1/2}q_1=\FL u +\sourceG,\\
 \dis \pa_t E-\na_x\times B=-G,\\
\dis \pa_t B+\na_x\times E =0,\\
\dis\na_x\cdot E =a_+-a_-,\ \ \na_x \cdot
B=0.
    \end{array}\right.
\end{equation}
Taking the Fourier transform in $x$  gives
\begin{equation}\label{ls-1f}
    \left\{\begin{array}{l}
  \dis     \pa_t \hat{u}+i\xi\cdot k \hat{u} - \hat{E}\cdot \xi \FM^{1/2}q_1=\FL \hat{u} +\hat{\sourceG},\\
 \dis \pa_t \hat{E}-i k\times \hat{B}=-\hat{G},\\
\dis \pa_t \hat{B}+i k\times \hat{E} =0,\\
\dis ik\cdot \hat{E} =\widehat{a_+-a_-},\ \ k\cdot
\hat{B}=0.
    \end{array}\right.
\end{equation}
Then equation \eqref{ls-1f}$_1$  implies
\begin{equation}
\notag
    \frac{1}{2}\pa_t \|\hat{u}\|_{L^2_\xi}^2-\rmre \int_{\R^3}(\FL \hat{u}\mid \hat{u})d\xi
    -\rmre (\hat{E}\mid \hat{G})=\rmre  {\int_{\R^3}(\hat{\sourceG}\mid \hat{u})d\xi}.
\end{equation}
We now use the vector identity
\begin{equation}
\notag
 {   (-i k\times \hat{B}\mid \hat{E})+(ik\times \hat{E}\mid \hat{B})=2 i\rmre\, (k\times \hat{E}\mid \hat{B}),}
\end{equation}
to observe from \eqref{ls-1f}$_2$-\eqref{ls-1f}$_3$ that
\begin{equation}
\notag
   \frac{1}{2}\pa_t (|\hat{E}|^2+|\hat{B}|^2)+\rmre (\hat{G}\mid \hat{E})=0.
\end{equation}
Since $(\hat{E}\mid \hat{G})$ and $(\hat{G}\mid \hat{E})$ have the
same real part, taking a summation of these two equalities gives
\begin{equation}
\notag
    \frac{1}{2}\pa_t \left(\|\hat{u}\|_{L^2_\xi}^2+|\hat{E}|^2+|\hat{B}|^2\right)-\rmre \int_{\R^3}(\FL \hat{u}\mid \hat{u})d\xi=\rmre  {\int_{\R^3}(\hat{\sourceG}\mid \hat{u})d\xi}.
\end{equation}
From \eqref{coerc} and $\sourceG=\{\FI-\FP\} \sourceG$,  one has
\begin{equation}\label{diss-micr}
\pa_t \left(\|\hat{u}\|_{L^2_\xi}^2+|\hat{E}|^2+|\hat{B}|^2\right)+\la \int_{\R^3}\nu(\xi)|\{\FI-\FP\}\hat{u}|^2d\xi\leq C\|\nu^{-1/2}\sourceG\|_{ {L^2_\xi}}^2,
\end{equation}
where we have used Cauchy's inequality in the form
$$
\rmre  {\int_{\R^3}(\hat{\sourceG}\mid \hat{u})d\xi}
=
\rmre  {\int_{\R^3}(\hat{\sourceG}\mid \{\FI-\FP\}\hat{u})d\xi}
\leq \frac{C}{\la}\|\nu^{-1/2}\hat{\sourceG}\|_{ {L^2_\xi}}^2
+
\lambda\|\nu^{1/2} \{\FI-\FP\}\hat{u}\|_{ {L^2_\xi}}^2,
$$
for a properly small constant $0<\la<\la_0$.

Here, we remark that equation \eqref{diss-micr} is the main estimate
for the construction of the time-frequency Lyapunov functional
$\CE(t,k)$. However, notice that for this time, the macroscopic part
$\FP\hat{u}$ and $\hat{E},\hat{B}$ are not included in the
dissipation rate of \eqref{diss-micr}. Next, based on the
macroscopic balance laws and high-order moment equation obtained in
Section \ref{sec.me}, we shall introduce some interactive functional
to capture the rest of the dissipation rate related to $\FP\hat{u}$ and
$\hat{E},\hat{B}$.


\subsubsection{Estimate on the macro dissipation}

Let us apply those computations in Section \ref{sec.me} to the
system \eqref{ls}$_1$-\eqref{ls}$_4$.
Taking the mean value of every
two equations with $\pm$ sign for \eqref{m0}, \eqref{m1}, \eqref{m2}
and noticing $\FP \sourceG=0$, one has
\begin{equation}\label{macro.1}
    \left\{
    \begin{array}{l}
      \dis \pa_t\left(\frac{a_++a_-}{2}\right)+\na_x\cdot b=0,\\
      \dis \pa_t b_i+\pa_i\left(\frac{a_++a_-}{2}+2c\right)+\frac{1}{2}\sum_{j=1}^3\pa_j\highG_{ij}(\{\FI-\FP\}u\cdot [1,1])=0,\\
      \dis \pa_t c+ \frac{1}{3}\na_x\cdot b +\frac{5}{6}\sum_{i=1}^3\partial_i \highB_i(\{\FI-\FP\}u\cdot [1,1])=0,
    \end{array}\right.
\end{equation}
for $1\leq i\leq 3$, where moment functions $\highG(\cdot)$ and
$\highB(\cdot)$ are defined in \eqref{def.gala},  {and we used the following facts
\begin{eqnarray*}
&&\langle \FM^{1/2},g_\pm \rangle =\langle ([1,0]+[0,1])\FM^{1/2}, g\rangle=0,\\
&&\langle \xi_i\FM^{1/2},g_+ + g_- \rangle =\langle {[\xi_i,\xi_i]}\FM^{1/2}, g\rangle=0,\ \ 1\leq i\leq 3,\\
&&\langle \frac{1}{6}(|\xi|^2-3)\FM^{1/2},g_+ + g_- \rangle =\langle [|\xi|^2,|\xi|^2]\FM^{1/2}, g\rangle=0,
\end{eqnarray*}
due to $\FP g=0$ and likewise for $\FL u=[\FL_+u, \FL_- u]$ due to $\FP \FL u=0$.} Similarly, it
follows from \eqref{m2ii}, \eqref{m2ij} and \eqref{m3} that
\begin{equation}\label{macro.2}
    \left\{
    \begin{array}{l}
      \dis \pa_t \left[\frac{1}{2}\highG_{ij}(\{\FI-\FP\}u\cdot [1,1]) +2c\de_{ij}\right]\\
      \dis \hspace{1.5cm}+\pa_i b_j+\pa_j b_i
      =
      \frac{1}{2}\highG_{ij}((l_++l_-)+(\sourceG_++\sourceG_-)),\\
      \dis \frac{1}{2}\pa_t \highB_i(\{\FI-\FP\}u\cdot [1,1])+\pa_i c=\frac{1}{2}\highB_i((l_++l_-)+(\sourceG_++\sourceG_-)),
    \end{array}\right.
\end{equation}
for $ 1\leq i,j\leq 3$, where $l_\pm$ is still defined in
\eqref{def.l}, and $\de_{ij}$ denotes as usual the Kronecker delta.

\begin{lemma}
There is a time-frequency functional  {$\CE_{\rm int}^{(1)}(t,k)$} defined by
\begin{eqnarray}
\notag
  \CE_{\rm int}^{(1)}(t,k) &=& \frac{1}{1+|k|^2}\sum_{i=1}^3\frac{1}{2} (i k_i \hat{c}\mid \highB_i(\{\FI-\FP\}\hat{u}\cdot [1,1]))\\
  \notag
  &&+\frac{\kappa_1}{1+|k|^2}\sum_{ {i,j}=1}^3(ik_i \hat{b}_j+ik_j\hat{b}_i\mid \frac{1}{2} \highG_{ij}(\{\FI-\FP\}\hat{u}\cdot [1,1])+2 {\hat{c}}\de_{ij})\\
  && +\frac{\kappa_2}{1+|k|^2}\sum_{i=1}^3\left(ik_i\frac{\hat{a}_++\hat{a}_-}{2}\mid \hat{b}_i\right),
  \label{def.int1}
\end{eqnarray}
with two properly chosen constants $0<\kappa_2\ll\kappa_1\ll 1$ such that
\begin{multline}
\dis \pa_t \rmre \CE_{\rm int}^{(1)}(t,k)+\frac{\la |k|^2}{1+|k|^2} \left(|\widehat{a_++a_-}|^2+|\hat{b}|^2+|\hat{c}|^2\right)\\
\dis \leq C(\|\{\FI-\FP\}\hat{u}\|_{L^2_\xi}^2+\|\nu^{-1/2}\hat{\sourceG}\|_{L^2_\xi}^2),
\label{diss-macro+}
\end{multline}
holds for any $t\geq 0$ and $k\in \R^3$.
\end{lemma}

\begin{proof}
The proof can be found in \cite[Lemma 4.1]{D-Hypo}.  Although
 \cite[Lemma 4.1]{D-Hypo} studies the pure Boltzmann equation without force terms our case is directly similar.  It follows from  \eqref{macro.1} and \eqref{macro.2}, which are otherwise not used in the sequel.
\end{proof}

In order to further obtain the dissipation rate related to
$\hat{a}_\pm$ from the formula
$$
|\hat{a}_+|^2+|\hat{a}_-|^2=\frac{|\widehat{a_++a_-}|^2}{2}+\frac{|\widehat{a_+-a_-}|^2}{2},
$$
we need to consider the dissipation of $\widehat{a_+-a_-}$. For
that, taking difference of two equations with $\pm$ sign for
\eqref{m0}, \eqref{m1} and also noticing  {$\FP \sourceG=0$ which implies $\langle \FM^{1/2}, g_\pm \rangle =0$}, one has
\begin{eqnarray}
&&\pa_t (a_+-a_-)+\na_x\cdot G=0,\label{m0-}\\
&&\pa_t G + \na_x (a_+-a_-)-2E+ \na_x\cdot \highG (\{\FI-\FP\}u\cdot q_1)\label{m1-}\\
\notag
&&\qquad\qquad\qquad\qquad\qquad\qquad\qquad= {\langle [\xi,-\xi] \FM^{1/2}, g+\FL \{\FI-\FP\} u \rangle}.
\end{eqnarray}
Note that here and hereafter
$
\left(\na_x\cdot \highG\right)_j (\cdot)
=
\partial_i \highG_{ij} (\cdot).
$
Together with
\begin{equation}\label{divE}
    \na_x\cdot E=a_+-a_-,
\end{equation}
one has the following lemma.


\begin{lemma}\label{lem.a}
For any $t\geq 0$ and $k\in \R^3$,
it holds that
\begin{equation}
\dis \frac{\pa_t \rmre (\hat{G}\mid ik \widehat{({a}_+-{a}_-)})}{(1+|k|^2)} +\la |\widehat{a_+-a_-}|^2\\
\dis \leq  {C(\|\{\FI-\FP\}\hat{u}\|_{L^2_\xi}^2+\|\nu^{-1/2}\hat{\sourceG}\|_{L^2_\xi}^2)}.  \label{lem.a-1}
\end{equation}
\end{lemma}


\begin{proof}
In fact, taking the
Fourier transform in $x$ for \eqref{m0-}, \eqref{m1-} and \eqref{divE} gives
\begin{equation}\label{lem.a-1.p1}
\left\{\begin{array}{l}
 \dis \pa_t \widehat{(a_+-a_-)}+ik\cdot \hat{G}=0,\\[3mm]
\dis \pa_t \hat{G}+ik \widehat{(a_+-a_-)}-2\widehat{E}+ik\cdot \highG(\{\FI-\FP\}\hat{u}\cdot q_1)\\[3mm]
\dis \qquad\qquad\qquad\qquad\qquad\qquad\qquad= {\langle [\xi,-\xi] \FM^{1/2}, \hat{g}+\FL \{\FI-\FP\} \hat{u} \rangle},\\[3mm]
\dis  ik\cdot \hat{E}=\widehat{a_+-a_-}.
\end{array}\right.
\end{equation}
On one hand, notice from \eqref{lem.a-1.p1}$_3$ that
\begin{equation}
\notag
    (ik \widehat{(a_+-a_-)}-2\widehat{E}\mid ik \widehat{(a_+-a_-)})=(|k|^2+2)|\widehat{a_+-a_-}|^2.
\end{equation}
On the other hand, it follows from \eqref{lem.a-1.p1}$_2$ that
\begin{eqnarray*}
&&  (ik \widehat{(a_+-a_-)}-2\widehat{E}\mid ik \widehat{(a_+-a_-)})\\
&& =(-\pa_t \hat{G} -ik\cdot \highG (\{\FI-\FP\}\hat{u}\cdot q_1) {+\langle [\xi,-\xi] \FM^{1/2}, \hat{g}+\FL \{\FI-\FP\} \hat{u} \rangle}\mid ik \widehat{(a_+-a_-)})\\
&&=-\pa_t (\hat{G}\mid ik \widehat{(a_+-a_-)})+(\hat{G}\mid ik \pa_t\widehat{(a_+-a_-)})\\
&&\ \ \ -(ik\cdot \highG(\{\FI-\FP\}\hat{u}\cdot q_1) {-\langle [\xi,-\xi] \FM^{1/2}, \hat{g}+\FL \{\FI-\FP\} \hat{u} \rangle}\mid ik \widehat{(a_+-a_-)}).
\end{eqnarray*}
Combining the above two equations
and using \eqref{lem.a-1.p1}$_1$, one has
\begin{multline*}
\pa_t (\hat{G}\mid ik \widehat{(a_+-a_-)})
+
(2+ |k|^2) ~ |\widehat{a_+-a_-}|^2
\\
=(\hat{G}\mid k ~k \cdot \hat{G})
 -(ik\cdot \highG(\{\FI-\FP\}\hat{u}\cdot q_1) {-\langle [\xi,-\xi] \FM^{1/2}, \hat{g}+\FL \{\FI-\FP\} \hat{u} \rangle}\mid ik \widehat{(a_+-a_-)}).
\end{multline*}
It follows using Cauchy's inequality that
\begin{multline*}
\pa_t\rmre (\hat{G}\mid ik \widehat{(a_+-a_-)})+\la (1+|k|^2)|\widehat{a_+-a_-}|^2\\
\leq  |k\cdot \hat{G}|^2+C|k\cdot \highG(\{\FI-\FP\}\hat{u}\cdot q_1)|^2
 {+C|\langle [\xi,-\xi]\FM^{1/2},\hat{g}\rangle|^2}\\
 {+C|\langle [\xi,-\xi]\FM^{1/2},\FL \{\FI-\FP\}\hat{u}\rangle|^2}\\
\leq  C {(1+|k|^2)}\|\{\FI-\FP\}\hat{u}\|_{L^2_\xi}^2 {+C\|\nu^{-1/2}\hat{g}\|_{L^2_\xi}^2}.
\end{multline*}
Therefore, \eqref{lem.a-1} holds by further dividing the above
estimate by $1+|k|^2$. Lemma \ref{lem.a} is proved.
\end{proof}


Notice that from \eqref{divE}, the dissipation rate in
\eqref{lem.a-1} can be rewritten as
$$
|\widehat{a_+-a_-}|^2=|k\cdot \hat{E}|^2=|k|^2|\tilde{k}\cdot
\hat{E}|.
$$
Here and in the sequel, we always denote $\tilde{k}=k/|k|$ for
$|k|\neq 0$. Therefore, for this time, the dissipation rate for
$\tilde{k}\times \hat{E}$ and $\tilde{k}\times \hat{B}$ is still not
included. Actually, they can be recovered from the Maxwell equations
of the electromagnetic field as well as the evolution equation
\eqref{m1-} of the linear coupling term $G$.

\subsubsection{Estimate on the electromagnetic dissipation}


As mentioned before, we now devote ourselves to obtaining the
dissipation rate related to $\tilde{k}\times \hat{E}$ and
$\tilde{k}\times \hat{B}$ in the following

\begin{lemma}\label{lem.eb}
For any $t\geq 0$ and $k\in \R^3$
it holds that
\begin{multline}
\dis \pa_t \left(\frac{\rmre (-ik\times \hat{B}\mid \hat{E})-|k|^2\rmre (\hat{G}\mid \hat{E})}{(1+|k|^2)^2} \right) +\la\frac{|k\times \hat{B}|^2}{(1+|k|^2)^2}\\
\dis \ \ \ \ \ +\frac{|k|^2}{(1+|k|^2)^2}|k\cdot \hat{E}|^2+\frac{\la |k|^2}{(1+|k|^2)^2}|\hat{E}|^2
\leq  {C(\|\{\FI-\FP\}\hat{u}\|_{L^2_\xi}^2+\|\nu^{-1/2}\hat{\sourceG}\|_{L^2_\xi}^2)}.
\label{lem.eb.1}
\end{multline}
\end{lemma}

\begin{proof}
Recall the Fourier transform in $x$ for the Maxwell system
\begin{equation}\label{lem.eb.p1}
    \left\{\begin{array}{l}
      \dis\pa_t \hat{E}-ik\times \hat{B}=-\hat{G},\\
       \dis \pa_t \hat{B}+ i k\times \hat{E}=0,\\
       \dis ik\cdot \hat{E}=\widehat{a_+-a_-},\ \ k\cdot \hat{B}=0.
    \end{array}\right.
\end{equation}
It follows that
\begin{eqnarray*}
|k\times \hat{B}|^2&=& (ik\times \hat{B}\mid \pa_t \hat{E}+\hat{G})\\
&=&\pa_t (ik\times \hat{B}\mid \hat{E})-(ik\times \pa_t\hat{B}\mid \hat{E}) +(ik\times \hat{B}\mid \hat{G}),
\end{eqnarray*}
where it further holds that
\begin{equation}
\notag
    -(ik\times \pa_t\hat{B}\mid \hat{E})=-( {k\times (k\times \hat{E})}\mid \hat{E})=|k\times \hat{E}|^2.
\end{equation}
Hence, one has the identity
\begin{equation}
\notag
    \pa_t (-ik\times \hat{B}\mid \hat{E})+|k\times \hat{B}|^2=|k\times \hat{E}|^2+(ik\times \hat{B}\mid \hat{G}),
\end{equation}
which using Cauchy's inequality, implies that
\begin{equation}\label{lem.eb.p2}
    \pa_t\left(\frac{\rmre (-ik\times \hat{B}\mid \hat{E})}{(1+|k|^2)^2}\right)+\la \frac{|k\times \hat{B}|^2}{(1+|k|^2)^2}\leq \frac{|k|^2}{(1+|k|^2)^2}|\hat{E}|^2+C|\hat{G}|^2.
\end{equation}
To control the first term on the r.h.s. of \eqref{lem.eb.p2}, one can again use \eqref{lem.a-1.p1}$_2$ together with
\eqref{lem.eb.p1}$_3$ and then \eqref{lem.eb.p1}$_1$ to get
\begin{multline*}
 -\pa_t(\hat{G}\mid \hat{E})+|k\cdot \hat{E}|^2+2|\hat{E}|^2\\
 =-(\hat{G}\mid \pa_t\hat{E})+(ik\cdot \highG(\{\FI-\FP\}\hat{u}\cdot q_1) {-\langle [\xi,-\xi] \FM^{1/2}, \hat{g}+\FL \{\FI-\FP\} \hat{u} \rangle}\mid  \hat{E})\\
 =|\hat{G}|^2-(\hat{G}\mid ik\times \hat{B})+(ik\cdot \highG(\{\FI-\FP\}\hat{u}\cdot q_1)\mid  \hat{E})\\
  {-(\langle [\xi,-\xi] \FM^{1/2}, \hat{g}+\FL \{\FI-\FP\} \hat{u} \rangle\mid  \hat{E})},
\end{multline*}
which implies
\begin{multline*}
\dis
- \pa_t\rmre (\hat{G}\mid \hat{E})+|k\cdot \hat{E}|^2+2|\hat{E}|^2\\
\dis \leq \eps (|\hat{B}|^2+|\hat{E}|^2) {+\frac{C}{\eps} \left[(1+|k|^2)|\hat{G}|^2+|k|^2|\highG(\{\FI-\FP\}\hat{u}\cdot q_1)|^2\right]}\\
 {+\frac{C}{\eps} \left(|\langle [\xi,-\xi]\FM^{1/2},\hat{g}\rangle|^2
+C|\langle [\xi,-\xi]\FM^{1/2},\FL \{\FI-\FP\}\hat{u}\rangle|^2\right)}\\
 {\leq \eps (|\hat{B}|^2+|\hat{E}|^2)
+\frac{C}{\eps}\left[(1+|k|^2)\|\{\FI-\FP\}\hat{u}\|_{L^2_\xi}^2+\|\nu^{-1/2}\hat{\sourceG}\|_{L^2_\xi}^2\right]},
\end{multline*}
for a constant $0<\eps\leq 1$ to be chosen later. Notice $|k|^2|\hat{B}|^2=|k\times \hat{B}|^2$ due to $k\cdot \hat{B}=0$. Then, further multiplying the above inequality by $|k|^2/(1+|k|^2)^2$ gives
\begin{multline}
\dis- \pa_t\frac{|k|^2\rmre (\hat{G}\mid \hat{E})}{(1+|k|^2)^2}+\frac{|k|^2}{(1+|k|^2)^2}|k\cdot \hat{E}|^2+\frac{(2-\eps)|k|^2}{(1+|k|^2)^2}|\hat{E}|^2\\
\dis \leq \eps \frac{|k\times \hat{B}|^2}{(1+|k|^2)^2}+\frac{C}{\eps} \frac{|k|^2}{1+|k|^2}( {\|\{\FI-\FP\}\hat{u}\|_{L^2_\xi}^2+\|\nu^{-1/2}\hat{\sourceG}\|_{L^2_\xi}^2}).\label{lem.eb.p3}
\end{multline}
Therefore, \eqref{lem.eb.1} follows from taking summation of \eqref{lem.eb.p2}, \eqref{lem.eb.p3} and then choosing $0<\eps<1$ small enough.  {Then} Lemma \ref{lem.eb} is proved.
\end{proof}


Here, we remark that although the pure homogeneous Maxwell system
usually preserves the energy, the electromagnetic field $[E,B]$ in
the Vlasov-Maxwell-Boltzmann system indeed has some kind of weak
dissipation which results essentially from the coupling of $[E,B]$
with the microscopic moment function $G$.

\subsubsection{Derivation of the  time-frequency Lyapunov inequality}

Now, we are in a position to prove

\begin{theorem}\label{thm.tfli}
Let $U=[u,E,B]$ be the solution to the Cauchy problem \eqref{ls}
with $\FP g=0$. Then there is a time-frequency functional $\CE(t,k)$
such that
\begin{equation}\label{thm.tfli.1}
    \CE(t,k)\sim \|\hat{u}\|_{L^2_\xi}^2+|[\hat{E},\hat{B}]|^2,
\end{equation}
where
$
|[\hat{E},\hat{B}]|^2
 \eqdef
 |\hat{E}|^2
 +
  |\hat{B}|^2
$
and for any $t\geq 0$ and $k\in \R^3$ we have
\begin{equation}\label{thm.tfli.2}
    \pa_t \CE(t,k)+\frac{\la |k|^2}{(1+|k|^2)^2} \CE(t,k)\leq C\|\nu^{-1/2}\hat{\sourceG}\|_{L^2_\xi}^2.
\end{equation}
\end{theorem}


\begin{proof}
Let
\begin{equation}\label{thm.tfli.p1}
    \CE(t,k)\eqdef \|\hat{u}\|_{L^2_\xi}^2+|[\hat{E},\hat{B}]|^2+\kappa_3\rmre (\CE_{\rm int}^{(1)}(t,k)+\CE_{\rm int}^{(2)}(t,k)),
\end{equation}
for a constant $\kappa_3>0$ to be determined later, where $\CE_{\rm int}^{(1)}(t,k)$ is given by \eqref{def.int1} and $\CE_{\rm int}^{(2)}(t,k)$ is denoted by
\begin{equation}\label{def.int2}
  \CE_{\rm int}^{(2)}(t,k) \eqdef \frac{( {\hat{G}}\mid ik \widehat{({a}_+-{a}_-)})}{(1+|k|^2)}+\frac{(-ik\times \hat{B}\mid \hat{E})-|k|^2(\hat{G}\mid \hat{E})}{(1+|k|^2)^2}.
\end{equation}
One can fix $\kappa_3>0$ small enough such that \eqref{thm.tfli.1} holds true. The rest is to check \eqref{thm.tfli.2}. In fact, the linear combination of \eqref{diss-micr}, \eqref{diss-macro+}, \eqref{lem.a-1} and \eqref{lem.eb.1} according to the definition \eqref{thm.tfli.p1}
implies
\begin{multline*}
\dis\pa_t \CE(t,k)+\la \|\nu^{1/2}\{\FI-\FP\}\hat{u}\|_{L^2_\xi}^2
+\frac{\la |k|^2}{1+|k|^2}(|\hat{a}_\pm|^2
+|\hat{b}|^2+|\hat{c}|^2)+\la  |k\cdot \hat{E}|^2
\\
\dis +\frac{\la |k|^2}{(1+|k|^2)^2} \left(|
\hat{E}|^2+|\tilde{k}\times \hat{B}|^2 \right) \leq
C\|\nu^{-1/2}\hat{\sourceG}\|_{L^2_\xi}^2,
\end{multline*}
that is
\begin{multline*}
\dis\pa_t \CE(t,k)+\la \|\nu^{1/2}\{\FI-\FP\}\hat{u}\|_{L^2_\xi}^2+\frac{\la |k|^2}{1+|k|^2}\|\FP \hat{u}\|_{L^2_\xi}^2+\la |k|^2 |\tilde{k}\cdot \hat{E}|^2
\\
\dis +\frac{\la |k|^2}{(1+|k|^2)^2} (|\tilde{k}\times \hat{E}|^2+|\tilde{k}\times \hat{B}|^2)\leq C\|\nu^{-1/2}\hat{\sourceG}\|_{L^2_\xi}^2,
\end{multline*}
since one has
$$
|\hat{a}_\pm|^2+|\hat{b}|^2+|\hat{c}|^2\sim \|\FP
\hat{u}\|_{L^2_\xi}^2.
$$
Noticing further that $|\hat{B}|^2=|\tilde{k}\times \hat{B}|^2$ due
to the fact that $B$ is divergence free,
\eqref{thm.tfli.2} follows.
\end{proof}


\subsection{Proof of time-decay of linear solutions}\label{sec.tf}
Our proof of Theorem \ref{thm.ls} is based on Theorem \ref{thm.tfli} and  {the further analysis of \eqref{thm.tfli.2} over the low and high frequency domains as follows. An alternative new time-frequency splitting method will be discussed at the end.}


\begin{proof}[Proof of Theorem \ref{thm.ls}]
Define the notations
\begin{equation}
\notag
    p(k)\eqdef
    \frac{\la |k|^2}{(1+|k|^2)^2},\quad
    \CE_{\sourceG}(t,k)\eqdef C\|\nu^{-1/2}\hat{\sourceG}\|_{L^2_\xi}^2.
\end{equation}
We rewrite \eqref{thm.tfli.2}, for any $t\geq 0$ and $k\in \R^3$, as
\begin{equation*}
    \pa_t \CE(t,k)+p(k) \CE(t,k)\leq \CE_{\sourceG}(t,k).
\end{equation*}
This implies from the Gronwall inequality, for any $t\geq 0$ and $k\in \R^3$, that
\begin{equation}\label{thm.ls.p02}
   \CE(t,k)\leq e^{-p(k)t}\CE(0,k)+\int_0^t e^{-p(k)(t-s)}  \CE_{\sourceG}(s,k)ds.
\end{equation}
Due to \eqref{thm.tfli.1}, notice that for any fixed $m\geq 0$ we have
\begin{equation}\label{thm.ls.p08}
    \|\na_x^m u(t)\|^2+\|\na_x^m[E(t),B(t)]\|^2\sim \int_{\R^3}|k|^{2m}\CE(t,k)dk.
\end{equation}
Now, to  first prove \eqref{thm.ls.1}, one can apply \eqref{thm.ls.p02} with $\sourceG=0$ and hence $\CE_{\sourceG}(t,k)=0$ to bound $\CE(t,k)$ as follows
\begin{equation}
\notag
 \int_{\R^3}|k|^{2m}\CE(t,k)dk\leq \left(\int_{|k|\leq 1}+\int_{|k|>1}\right)|k|^{2m}e^{-p(k)t}\CE(0,k)dk.
\end{equation}
Here, notice that for $|k|\leq 1$,
\begin{equation}
\notag
    p(k)=\frac{\la |k|^2}{(1+|k|^2)^2}\geq \frac{\la}{4}|k|^2,
\end{equation}
and for $|k|\geq 1$,
\begin{equation}  \notag
    p(k)\geq \frac{\la}{4|k|^2}.
\end{equation}
With these two estimates we have the upper bound
\begin{multline*}
 \int_{\R^3}|k|^{2m}\CE(t,k)dk \leq   \int_{|k|\leq 1}|k|^{2m}e^{-\frac{\la}{4}|k|^2t}\CE(0,k)dk\\
 +\int_{|k|\geq 1}|k|^{2m}e^{-\frac{\la t}{4|k|^2}}\CE(0,k)dk.
\end{multline*}
Here, as in \cite{Ka,Ka-BE13}, from the H\"{o}lder and Hausdorff-Young inequalities, the integration over $|k|\leq 1$ is bounded as
$$
 \int_{|k|\leq 1}|k|^{2m}e^{-\frac{\la}{4}|k|^2t}\CE(0,k)dk
 \leq  C(1+t)^{-3(\frac{1}{r}-\frac{1}{2})-m}\left(
\|u_0\|_{Z_r}^2+\|[E_0,B_0]\|_{L^r_x}^2\right),
$$
for $1\leq r \leq 2$.
See also, for instance, \cite{DS-VPB}.

The other integration over $|k|\geq 1$ is estimated as
\begin{equation}
\notag
 \int_{|k|\geq 1}|k|^{2m}e^{-\frac{\la t}{4|k|^2}}\CE(0,k)dk\leq  \int_{|k|\geq 1} |k|^{ {2m+2\ell}} \CE(0,k)dk\,
\sup_{|k|\geq 1} \frac{1}{|k|^{2\ell}}e^{-\frac{\la t}{4|k|^2}}.
\end{equation}
Since
\begin{equation}
\notag
    \sup_{|k|\geq 1} \frac{1}{|k|^{2\ell}}e^{-\frac{\la t}{4|k|^2}}\leq C_{\la,\ell}(1+t)^{-\ell},
\end{equation}
it follows that
\begin{equation}  \notag
 \int_{|k|\geq 1}|k|^{ {2m}}e^{-\frac{\la t}{4|k|^2}}\CE(0,k)dk\leq  C(1+t)^{-\ell}(\|\na_x^{m+\ell}u_0\|^2+\|\na_x^{m+\ell} [E_0,B_0]\|^2).
\end{equation}
Collecting the above estimates as well as \eqref{thm.ls.p08} gives \eqref{thm.ls.1}.

Similarly, to prove \eqref{thm.ls.2}, one can apply \eqref{thm.ls.p02} with $U_0=0$ and hence $\CE(0,k)=0$ to bound $\CE(t,k)$ as in the following
\begin{multline}
 \nonumber
  \int_{\R^3}|k|^{2m}\CE(t,k)dk
 \leq   \int_{\R^3}|k|^{2m}\left[\int_0^t e^{-p(k)(t-s)}  \CE_{\sourceG}(s,k)ds\right]dk
 \\
 =\int_0^t\left[\int_{\R^3}|k|^{2m} e^{-p(k)(t-s)}\CE_{\sourceG}(s,k)dk\right]ds.
\end{multline}
Recall $\CE_{\sourceG}(s,k)= C\|\nu^{-1/2}\hat{\sourceG}(s,k)\|_{L^2_\xi}^2$. Then, similarly as before, it follows that
\begin{multline*}
\int_{\R^3}|k|^{2m} e^{-p(k)(t-s)}\CE_{\sourceG}(s,k)dk
 \leq C (1+t-s)^{-3(\frac{1}{r}-\frac{1}{2})-m}\|\nu^{-1/2}\sourceG(s)\|_{Z_r}^2\\
 +C (1+t-s)^{-\ell} \|\nu^{-1/2} \na_x^{m+\ell}\sourceG(s)\|^2.
\end{multline*}
The above two estimates together with \eqref{thm.ls.p08} give \eqref{thm.ls.2}.
\end{proof}

\begin{remark}\label{time.f.m}
Our initial proof of the time-decay for linear solutions involved a new time-frequency splitting method.  The key point was to consider the sets
$$
\{ t \ge \lambda |k| \},
\quad
\{ t < \lambda |k| \},
$$
separately in the energy estimates.  In this approach, when $t$ is smaller than $|k|$ one can directly pay for the time decay by placing additional regularity on the initial data.  And when $t$ is larger than $|k|$ there is a gain to be exploited using the time weighted estimates.  This approach costs a few extra pages,  it could however be more robust for other systems in which derivative loss is present.
\end{remark}

\section{Energy estimates on the nonlinear system}\label{sec.energy}


The goal of this section is to prove some nonlinear energy estimates, specifically \eqref{prop.exi.3} in Proposition \ref{prop.exi} and
\eqref{thm.energy.1}, \eqref{thm.energy.2} in Theorem
\ref{thm.energy}. Combining the time-decay property of the
linearized system studied in Section \ref{sec.decayl}, these
nonlinear energy estimates can be used to deduce the time rates of
the corresponding energy functionals in \eqref{thm.energy.1} and
\eqref{thm.energy.2}, which will be shown in the next section.

\subsection{Dissipation on the electromagnetic field}

In this subsection, we prove Proposition \ref{prop.exi}. It suffices
to prove \eqref{prop.exi.3} due to \cite[Theorem 1]{S.VMB}. In fact,
from the proof of Theorem 1 in \cite{S.VMB}, there is $\CE_N(t)$
such that as long as $\CE_N(0)$ is small enough, one has
\begin{eqnarray}
\notag
\frac{d}{dt}\CE_N(t)&+&\sum_{1\leq |\al|\leq N}\|\nu^{\frac{1}{2}}\pa^\al u(t)\|^2\\
&+&\sum_{|\al|+|\be|\leq N}\|\nu^{\frac{1}{2}}\pa^\al_\be \{\FI-\FP\}u(t)\|^2+ \la \|E(t)\|^2\leq 0,\label{app.p1}
\end{eqnarray}
for any $t\geq 0$. Recall the definition \eqref{def.dNm} of $\CD_N(t)=\CD_{N,0}(t)$, the rest is to prove that the additional term
\begin{equation}  \notag
    \sum_{1\leq |\al|\leq N-1} \|\pa^\al [E(t),B(t)]\|^2,
\end{equation}
can be included in the dissipation rate in \eqref{app.p1}. Due to \cite[Lemma 6]{S.VMB} and the uniform-in-time smallness of $\CE_N(t)$ for  solutions,
\begin{equation}\label{app.p2}
    \sum_{|\al|\leq N-1}\|\pa^\al E\|^2\leq C\|\{\FI-\FP\} u\|^2+C\sum_{1\leq |\al|\leq N}\|\pa^\al u\|^2.
\end{equation}
 {This estimate can also be obtained from the following macroscopic balance law deduced from the difference of \eqref{m1.non} with $\pm$ sign:
\begin{multline*}
\pa_t \langle[\xi,-\xi]\FM^{1/2}, \{\FI-\FP\}u\rangle + {\na_x}  (a_+ - a_ -) -2E + {\na_x \cdot \Theta} (\{\FI-\FP\}u\cdot q_1)\\
=E (a_+ + a_-)+2b\times B +\langle [\xi,-\xi]\FM^{1/2},\FL u +\Ga(u,u)\rangle,
\end{multline*}
that is,
\begin{multline*}
2E=\pa_t \langle[\xi,-\xi]\FM^{1/2}, \{\FI-\FP\}u\rangle + {\na_x}  (a_+ - a_ -)  + {\na_x \cdot  \Theta } (\{\FI-\FP\}u\cdot q_1)\\
-E (a_+ + a_-)-2b\times B -\langle [\xi,-\xi]\FM^{1/2},\FL \{\FI-\FP\}u +\Ga(u,u)\rangle.
\end{multline*}
Since $E$ is the sum of some zero-order microscopic term of $\{\FI-\FP\}u$, first-order derivatives and quadratically nonlinear terms,  \eqref{app.p2} follows from the above representation of $E$, and it will be also refined in Lemma \ref{lem.bdd.EB} later on.} The dissipation estimate on $B(t)$ follows from that of $E(t)$ and the Maxwell system \eqref{VMB.pM1}-\eqref{VMB.pM2}. Recall
\begin{equation}\label{maxwell.u}
\begin{split}
&\pa_t E-\na_x\times B=-\langle[\xi,-\xi]\FM^{1/2}, \{\FI-\FP\}u\rangle,\\
&\pa_t B+\na_x\times E=0.
\end{split}
\end{equation}
Take $\al=(\al_0,\al_1,\al_2,\al_3)$ with $1\leq |\al|\leq N-1$. If
$\al_0>0$, one can use
\begin{equation}\label{max.B.est}
    \pa^\al B=\pa^{\al'}\pa_tB=-\pa^{\al'}\na_x\times E,
\end{equation}
for some $\al'$ with $|\al'|=|\al|-1$ to obtain
\begin{equation}  \notag
    \|\pa^\al B\|^2\leq \sum_{1\leq |\al|\leq N-1}\|\pa^\al E\|^2.
\end{equation}
Otherwise, when $\al_0=0$,
\begin{eqnarray*}
 \pa^\al B=\pa^{\al'}\pa_iB&=&-\pa^{\al'}\pa_i\De_x^{-1} {\big(\na_x\times(\na_x\times B)\big)}\\
 &=&-\pa^{\al'}\pa_i\De_x^{-1}\na_x\times [\pa_t E+\langle[\xi,-\xi]\FM^{1/2}, \{\FI-\FP\}u\rangle],
\end{eqnarray*}
for some $\al'$ and $1\leq i\leq 3$ with $|\al'|=|\al|-1$, where we used the vector identity
 {
$\De_x B=-\na_x\times \big(\na_x\times B\big)$}
since $B$ is divergence free. Notice that $\pa_i\De_x^{-1}\na_x$ is bounded from $L^p_x$ to itself for any $1<p<\infty$. Thus, in the case when $\al_0=0$, one has
\begin{equation} \notag
     \|\pa^\al B\|^2\leq C\sum_{1\leq |\al|\leq N-1}\|\pa^\al E\|^2 +C\sum_{|\al|\leq N-2}\| {\pa^\al}\{\FI-\FP\}u\|^2.
\end{equation}
Therefore, taking summation of estimates on $B$ over $1\leq |\al|\leq N-1$ gives
\begin{equation}\label{app.p3}
   \sum_{1\leq |\al|\leq N-1}  \|\pa^\al B\|^2\leq C\sum_{1\leq |\al|\leq N-1}\|\pa^\al E\|^2+C\sum_{|\al|\leq N-2}\| {\pa^\al}\{\FI-\FP\}u\|^2.
\end{equation}
Then \eqref{prop.exi.3} follows from the proper linear combination of \eqref{app.p1} together with \eqref{app.p2} and \eqref{app.p3}. \qed

\subsection{Velocity-weighted energy estimates} In this subsection, we are concerned with velocity-weighted energy estimates in the form of \eqref{thm.energy.1} for solutions to the nonlinear system \eqref{VMB.pB}-\eqref{VMB.pM3}. Recall the first part of Theorem \ref{thm.energy} in the following:

\medskip

\noindent{\bf Claim:} {\it  For any $m=0,1,2,\ldots$, there are
$\CE_{N,m}(t)$ and $\CD_{N,m}(t)$ such that if $\CE_{N,m-1}(0)$ is sufficiently small and $\CE_{N,m}(0)$ is finite then
\begin{eqnarray}
\frac{d}{dt}\CE_{N,m}(t)+\la \CD_{N,m}(t)\leq 0,
\label{app.vw.p01}
\end{eqnarray}
holds for any $t\geq 0$, where $\CE_{N,-1}(0)=\CE_{N}(0)$ is set and
$\la$ may depend on $m$.}

\medskip

The claim above can be proved by induction on $m\geq 0$. First, it
is obvious that due to Proposition \ref{prop.exi},  there is
$\CE_{N,0}(t)$ such that \eqref{app.vw.p01} holds for $m=0$ if
$\CE_{N,0}(0)$ is sufficiently small. Now, supposing that the claim
is true for some $m\geq 0$, we shall prove that it also holds for
$m+1$. For that, assume that $\CE_{N,m}(0)$ is sufficiently small.
Since $\CE_{N,m-1}(0)\leq C \CE_{N,m}(0)$, then $\CE_{N,m-1}(0)$ is
also sufficiently small. It then follows from the induction
assumption that
\begin{equation}\label{app.vw.p02}
    \frac{d}{dt}\CE_{N,m}(t)+\la \CD_{N,m}(t)\leq 0.
\end{equation}
This implies that $\CE_{N,m}(t)$ is non-increasing in $t$, and hence $\CE_{N,m}(t)$ is  sufficiently small uniformly in time. Now, our goal is to prove that there are $\CE_{N,m+1}(t)$, $\CD_{N,m+1}(t)$ such that if $\CE_{N,m+1}(0)$ is finite, then
\begin{equation}\label{app.vw.p03}
    \frac{d}{dt}\CE_{N,m+1}(t)+\la \CD_{N,m+1}(t)\leq 0,
\end{equation}
for any $t\geq 0$. We directly define $\CD_{N,m+1}(t)$ by \eqref{def.dNm}. Then, it remains to construct $\CE_{N,m+1}(t)$. In fact, we carry it out  along the lines of \cite[Subsection 4.1]{DS-VPB}
and \cite{S.VMB}.

\medskip

\noindent{\it Step 1.} We split the solution $u$ to equation \eqref{VMB.pB} into $u=\FP u + \{\FI-\FP\}u$ and take
$\{\FI-\FP\}$ of the resulting equation to obtain
\begin{multline}
\pa_t \{\FI-\FP\}u + \xi \cdot \na_x \{\FI-\FP\}u + q(E+\xi\times B)\cdot \na_\xi \{\FI-\FP\}u \label{app.vw.p04}
\\
=
\FL\{\FI-\FP\} u +\Ga(u,u)+\frac{q}{2}E\cdot \xi \{\FI-\FP\}u+\{\FI-\FP\}(E\cdot \xi\FM^{1/2}q_1)
 \\
-\{\FI-\FP\}(\xi \cdot \na_x \FP u + q(E+\xi\times B)\cdot \na_\xi \FP u-\frac{q}{2}E\cdot \xi \FP u)
\\
 +\FP (\xi \cdot \na_x \{\FI-\FP\}u + q(E+\xi\times B)\cdot \na_\xi \{\FI-\FP\}u-\frac{q}{2}E\cdot \xi \{\FI-\FP\}u).
\end{multline}
Multiplying the above equation by $\nu^{m+1}\{\FI-\FP\}u$ and  {integrating} in $x,\xi$:
\begin{equation}
\label{app.vw.p05.int1}
\frac{1}{2}\frac{d}{dt}\|\nu^{\frac{m+1}{2}}\{\FI-\FP\}u\|^2 -
 \langle \nu^{m+1}\FL\{\FI-\FP\} u,\{\FI-\FP\} u\rangle
 =
 \Ga_1+ \Ga_2 {+\Ga_3},
\end{equation}
where $ \Ga_1 =  \langle \nu^{m+1} \Ga(u,u),\{\FI-\FP\} u\rangle$,
and
\begin{multline*}
 \Ga_2 =
\left\langle \frac{q}{2}E\cdot \xi \{\FI-\FP\}u+\{\FI-\FP\}(E\cdot \xi\FM^{1/2}q_1),
\nu^{m+1}  \{\FI-\FP\}u
\right\rangle
 \\
-\left\langle
\{\FI-\FP\}(\xi \cdot \na_x \FP u + q(E+\xi\times B)\cdot \na_\xi \FP u-\frac{q}{2}E\cdot \xi \FP u),
\nu^{m+1}  \{\FI-\FP\}u
\right\rangle
\\
 +\left\langle
 \FP (\xi \cdot \na_x \{\FI-\FP\}u + q(E+\xi\times B)\cdot \na_\xi \{\FI-\FP\}u),
 \nu^{m+1}  \{\FI-\FP\}u  \right\rangle
 \\
 -\left\langle
 \FP \left(\frac{q}{2}E\cdot \xi \{\FI-\FP\}u\right),
 \nu^{m+1}  \{\FI-\FP\}u
\right\rangle,
\end{multline*}
 {and}
\begin{equation*}
     {\Ga_3=\langle -q(E+\xi\times B)\cdot \na_\xi \{\FI-\FP\}u),\nu^{m+1}  \{\FI-\FP\}u  \rangle.}
\end{equation*}
Here and in the sequel, for simplicity of notations, we also use
$\langle\cdot,\cdot\rangle$ to denote the inner product over
$L^2_{x,\xi}$ without any confusion. We will estimate each of the
three terms in \eqref{app.vw.p05.int1}.

These estimates rely on three important observations. The first observation is that from \cite[Lemma 7]{S.VMB} (except for the presence of the momentum weight)
and also \cite[Lemma 3.3]{DUYZ} (to handle the momentum weight) we have
$$
\left| \Ga_1 \right| \le
C\sqrt{\CE_{N,m}(t)}\CD_{N,m+1}(t).
$$
The second key observation is that
$$
 - \langle \nu^{m+1}\FL\{\FI-\FP\} u,\{\FI-\FP\} u\rangle
 \ge \la \|\nu^{\frac{m+2}{2}}\{\FI-\FP\}u\|^2
 -
 C_\lambda  \|\{\FI-\FP\}u\|^2.
$$
This follows from $\FL = - \nu + K$, and the standard compact interpolation estimate for $K$:
$
\left| \langle \nu^{m+1}K f, f\rangle \right|
 \le \eta \|\nu^{\frac{m+2}{2}}f\|^2
 +
 C_\eta  \|f\|^2,
$
which holds for any small $\eta >0$.

The third useful observation is the Sobolev embedding trick used for instance in \cite{S.VMB}.
Specifically
we combine the $L^6({\mathbb R}^3)$ Sobolev inequality for gradients
with the embedding $W^{1,6}({\mathbb R}^3)\subset L^\infty ({\mathbb R}^3)$
 to obtain
\begin{multline}
\esssup_{x\in\mathbb{R}^3} \int_{{\mathbb R}^3} |u(x,\xi)|^2d\xi \le
C\int_{{\mathbb R}^3} \|u \|^2_{L^\infty({\mathbb R}^3_x)}d\xi
\\
\le C\int_{{\mathbb R}^3} \|\nabla_x u\|^2_{H^1({\mathbb
R}^3_x)}d\xi = C \|\nabla_x u\|^2_{\mathcal{H}^{1}}.
\label{sobolev.trick}
\end{multline}
This allows us to control all of the cubic terms in $\Ga_2$ without derivatives.
 {
We will also use this estimate when only the spatial variables $\R^3_x$ are present.
}

Using   {Cauchy-Schwarz}, we easily obtain the bound
 {
\begin{multline*}
\left|  \Ga_2  \right|
\le
C\|E \|_{L^\infty({\mathbb R}^3_{x})}\|\nu^{\frac{m+2}{2}}\{\FI-\FP\}u\|^2
+
C\|E \| \|\nu^{\frac{1}{2}}\{\FI-\FP\}u\|
\\
+
C\| \na_x \FP u \| \|\nu^{\frac{1}{2}}\{\FI-\FP\}u\|
+
C\|[E,B] \|_{L^\infty({\mathbb R}^3_{x})}
\|\FP u\|
\|\nu^{\frac{1}{2}}\{\FI-\FP\}u\|
\\
+C \|\nu^{\frac{1}{2}}\na_x \{\FI-\FP\}u\| \|\nu^{\frac{1}{2}}\{\FI-\FP\}u\|\\
+C\|[E,B] \|_{L^\infty({\mathbb R}^3_{x})}
\|\nu^{\frac{m+2}{2}}\na_\xi \{\FI-\FP\}u\|  \|\nu^{\frac{m+2}{2}}\{\FI-\FP\}u\|
\\
 +
 \|E \|_{L^\infty({\mathbb R}^3_{x})}
 \|\nu^{\frac{1}{2}}\{\FI-\FP\}u\|^2.
\end{multline*}
In several places in the estimate above, the velocity growth, $\xi$, is absorbed by the projections $\FP$.  Then it follows from this last estimate and \eqref{sobolev.trick} that
}
$$
\left| \Ga_2 \right|
\le
 C\CD_{N,m}(t)+ C\sqrt{\CE_{N,m}(t)}\CD_{N,m+1}(t).
$$
 {To estimate $\Ga_3$, notice}
\begin{equation*}
     {\na_\xi\{\FI-\FP\}u=(\na_\xi \nu^{-\frac{m+1}{2}}) \nu^{\frac{m+1}{2}}\{\FI-\FP\}u
    +\nu^{-\frac{m+1}{2}}\na_\xi( \nu^{\frac{m+1}{2}}\{\FI-\FP\}u),}
\end{equation*}
 {where the {second term on the right contributes} nothing into $\Ga_3$. Then, one has}
\begin{multline*}
 { \Ga_3 =  \langle -q(E+\xi\times B)\cdot (\na_\xi \nu^{-\frac{m+1}{2}})\nu^{\frac{m+1}{2}}\{\FI-\FP\}u,\nu^{m+1}  \{\FI-\FP\}u  \rangle}\\
  {  =\langle \frac{q(m+1)(E+\xi\times B)}{2\nu}\cdot (\na_\xi \nu)\nu^{\frac{m+1}{2}}\{\FI-\FP\}u,\nu^{\frac{m+1}{2}} \{\FI-\FP\}u  \rangle}\\
  {  \leq  C\|[E,B]\|_{L^\infty(\R^3_x)}\|\nu^{\frac{m+1}{2}} \{\FI-\FP\}u\|^2\leq C\sqrt{\CE_{N,m}(t)}\CD_{N,m}(t).}
\end{multline*}
 {Note that we have used  $\nu(\xi)\sim (1+|\xi|^2)^{1/2}$ and
$\left|\na_\xi \nu(\xi)
\right|
\le C$.}
We collect the estimates in this section to achieve the final estimate of
\begin{eqnarray}
\frac{1}{2}\frac{d}{dt}\|\nu^{\frac{m+1}{2}}\{\FI-\FP\}u\|^2&+&\la
\|\nu^{\frac{m+2}{2}}\{\FI-\FP\}u\|^2
\label{app.vw.p05}\\
&\leq& C\CD_{N,m}(t)+ C\sqrt{\CE_{N,m}(t)}\CD_{N,m+1}(t).
\notag
\end{eqnarray}
This completes the first step in our proof of \eqref{app.vw.p01}.
\medskip

\noindent{\it Step 2.} Let $1\leq |\al|\leq N$. Applying $\pa^\al$ to \eqref{VMB.pB} and writing $\FL=-\nu+K$, one has
\begin{multline*}
\pa_t\pa^\al u+\xi\cdot \na_x\pa^\al u +q(E+\xi\times B)\cdot \na_\xi \pa^\al u+\nu\pa^\al u\\
=K\pa^\al u+\pa^\al\Ga(u,u)+\frac{q}{2}\pa^\al(E\cdot\xi u)+\pa^\al E\cdot \xi \FM^{1/2}q_1\\
 - q\comml \pa^\al, (E+\xi\times B)\cdot \na_\xi \commr u,
\end{multline*}
where $\comml\,\cdot , \, \cdot \commr$ denotes the usual  commutator.  Multiplying the above equation by $\nu^{m+1}\pa^\al u$, and taking integrations in $x,\xi$ one has
$$
\frac{1}{2}\frac{d}{dt}\|\nu^{\frac{m+1}{2}}\pa^\al u\|^2+\|\nu^{\frac{m+2}{2}}\pa^\al u\|^2
= J_1 + J_2.
$$
In this expression we  have used
\begin{eqnarray*}
 J_1 &\eqdef&
 \left\langle K\pa^\al u+\pa^\al\Ga(u,u), \nu^{m+1}\pa^\al u \right\rangle
 +
 \left\langle  (E+\xi\times B)\cdot \na_\xi \nu^{m+1}, |q\pa^\al u|^2 \right\rangle,
 \\
 J_2  &\eqdef&
  \left\langle \frac{q}{2}\pa^\al(E\cdot\xi u)+\pa^\al E\cdot \xi \FM^{1/2}q_1  - q\comml \pa^\al, (E+\xi\times B)\cdot \na_\xi \commr u, \nu^{m+1}\pa^\al u \right\rangle.
\end{eqnarray*}
Now using the estimates from Step 1,  {Cauchy-Schwarz}, and the Sobolev embedding \eqref{sobolev.trick}, since $|\alpha| \ge 1$, one has the estimates
$$
\left| J_1 \right| + \left| J_2 \right|
\le
C\CD_{N,m}(t)+ C\sqrt{\CE_{N,m}(t)}\CD_{N,m+1}(t).
$$
We collect these estimates, and take summation over $1\leq |\al|\leq N$, to obtain
\begin{eqnarray}
\frac{1}{2}\frac{d}{dt}\sum_{1\leq |\al|\leq N}\|\nu^{\frac{m+1}{2}}\pa^\al u\|^2&+&\la \sum_{1\leq |\al|\leq N}\|\nu^{\frac{m+2}{2}}\pa^\al u\|^2\label{app.vw.p06}\\
&\leq& C\CD_{N,m}(t)+ C\sqrt{\CE_{N,m}(t)}\CD_{N,m+1}(t).
\notag
\end{eqnarray}
This is the main energy inequality in the second step in our proof of \eqref{app.vw.p01}.

\medskip

\noindent{\it Step 3.}  Let $|\al|+|\be|\leq N$ with $|\be|\geq 1$. Applying $\pa^\al_\be$ to \eqref{app.vw.p04} and writing $\FL=-\nu+K$, we observe that $v=\pa^\al_\be \{\FI-\FP\} u$ satisfies
\begin{equation}
\pa_t v+\xi\cdot \na_x v +q(E+\xi\cdot B)\cdot \na_\xi v + \nu v=I_1+I_2+I_3,\label{app.vw.p07}
\end{equation}
where $I_i$, $i=1,2,3$,  {are defined by}
\begin{eqnarray*}
I_1&=&\pa_\be K\left(\pa^\al \{\FI-\FP\}u\right)
+
\pa_\be^\al \Ga(u,u)
\\
&&+\frac{q}{2}\pa_\be^\al(E\cdot \xi \{\FI-\FP\}u)+\pa_\be\{\FI-\FP\}(\pa^\al E\cdot \xi \FM^{1/2}q_1),\\
I_2&=&-\pa_\be^\al\{\FI-\FP\}(\xi \cdot \na_x \FP u + q(E+\xi\times B)\cdot \na_\xi \FP u-\frac{q}{2}E\cdot \xi \FP u)\nonumber \\
&&+\pa_\be^\al\FP (\xi \cdot \na_x \{\FI-\FP\}u + q(E+\xi\times B)\cdot \na_\xi \{\FI-\FP\}u
-\frac{q}{2}E\cdot \xi \{\FI-\FP\}u),
\end{eqnarray*}
and
\begin{multline*}
  I_3 = -\comml \pa_\be,\xi\cdot \na_x \commr \pa^\al \{\FI-\FP\} u -\comml \pa_\be, \nu(\xi) \commr \pa^\al\{\FI-\FP\} u  \\
  - q\comml \pa_\be^\al, (E+\xi\times B)\cdot \na_\xi \commr \{\FI-\FP\} u.
\end{multline*}
Multiplying \eqref{app.vw.p07} by $\nu^{m+1} v$, and integrating over $x,\xi$ we have
$$
\frac{1}{2}\frac{d}{dt}\|\nu^{\frac{m+1}{2}}\pa^\al_\be \{\FI-\FP\} u\|^2+\|\nu^{\frac{m+2}{2}}\pa^\al_\be \{\FI-\FP\} u\|^2
=
\sum_{j=1}^4\tilde{I}_j.
$$
Above $\tilde{I}_j \eqdef \langle I_j, \nu^{m+1} v\rangle$ for $j=1,2,3$ and
$
\tilde{I}_4 \eqdef  \left\langle  (E+\xi\times B)\cdot \na_\xi \nu^{m+1}, |q v|^2 \right\rangle.
$
Exactly as in the analogous estimate in Step 2, we have
$$
 {
\left| \tilde{I}_4 \right|
\le
 C\sqrt{\CE_{N,m}(t)}\CD_{N,m+1}(t).
 }
$$
From, for instance,
\cite[Lemma 2.1]{Guo2}, we see that
$
\left|
\pa_\be \nu(v)
\right|
\le C
$
so that
$$
\left| \tilde{I}_3 \right|
\le
C\CD_{N,m}(t)+ C\sqrt{\CE_{N,m}(t)}\CD_{N,m+1}(t).
$$
For this we used  {Cauchy-Schwarz} and the Sobolev embedding, as usual.

We may use the estimate such as \cite[Lemma 2.2]{Guo2} to see that $\forall \eta >0$:
\begin{multline*}
\left|
\left\langle  \pa_\be K\left(\pa^\al \{\FI-\FP\}u\right), \nu^{m+1} \pa^\al_\be \{\FI-\FP\} u \right\rangle
\right|
\le
\eta \sum_{|\be'|= |\beta|}\|\nu^{\frac{ {m+2}}{2}}\pa^\al_{\be'} \{\FI-\FP\} u\|^2
\\
+C_\eta \|\nu^{\frac{m+1}{2}}\pa^\al \{\FI-\FP\} u\|^2.
\end{multline*}
Strictly speaking, in \cite[Lemma 2.2]{Guo2} there is no velocity weight: $\nu^{m+1}$.  However this can be added to the proof directly without difficulty.  Furthermore,
$$
\left|
\left\langle  \pa_\be^\al \Ga(u,u), \nu^{m+1} \pa^\al_\be \{\FI-\FP\} u \right\rangle
\right|
\le
C\sqrt{\CE_{N,m}(t)}\CD_{N,m+1}(t).
$$
This is the content of \cite[Lemma 7]{S.VMB}, the same comment applies for the weight.  With these estimates, similar to the estimates in Step 2, we have that
$$
\left| \tilde{I}_1 \right| + {\left| \tilde{I}_2 \right|}
\le
C\CD_{N,m}(t)+ C\sqrt{\CE_{N,m}(t)}\CD_{N,m+1}(t)
+
\eta \sum_{|\be'|= |\beta|}\|\nu^{\frac{ {m+2}}{2}}\pa^\al_{\be'} \{\FI-\FP\} u\|^2.
$$
Since $\eta >0$ can be taken arbitrarily small, we add together each of these estimates
and further take a summation over $|\al|+|\be|\leq N$ with $|\be|\geq 1$ to obtain
\begin{multline}
\frac{1}{2}\frac{d}{dt}\sum_{\substack{|\al|+|\be|\leq N\\ |\be|\geq 1}} {C_{\al,\be}}\|\nu^{\frac{m+1}{2}}\pa^\al_\be \{\FI-\FP\} u\|^2+\la \sum_{\substack{|\al|+|\be|\leq N\\ |\be|\geq 1}}\|\nu^{\frac{m+2}{2}}\pa^\al_\be \{\FI-\FP\} u\|^2\label{app.vw.p08}\\
\leq
C\CD_{N,m}(t)+ C\sqrt{\CE_{N,m}(t)}\CD_{N,m+1}(t),
\end{multline}
 {where $C_{\al,\be}$ are some positive constants.} This is the third and final estimate which we need to prove \eqref{app.vw.p01}.

\medskip

\noindent Now, let us define
\begin{eqnarray*}
\CE_{N,m+1}(t)&=& \CE_{N,m}(t)+\kappa_1 \|\nu^{\frac{m+1}{2}}\{\FI-\FP\}u\|^2+\kappa_1 \sum_{1\leq |\al|\leq N}\|\nu^{\frac{m+1}{2}}\pa^\al u\|^2\\
&&+\kappa_2 \sum_{\substack{|\al|+|\be|\leq N\\ |\be|\geq 1}} {C_{\al,\be}}\|\nu^{\frac{m+1}{2}}\pa^\al_\be \{\FI-\FP\} u\|^2,
\end{eqnarray*}
for properly small constants  {$0< \kappa_2\ll \kappa_1\ll 1$} to be chosen later.  Notice that for $0<\kappa_1, \kappa_2<1$, \eqref{def.eNm} holds true for $m+1$. By letting
 {$0< \kappa_2\ll \kappa_1\ll 1$} be  small enough, the sum of  \eqref{app.vw.p02} and \eqref{app.vw.p05}$\times \kappa_1$, \eqref{app.vw.p06}$\times \kappa_1$, \eqref{app.vw.p08}$\times \kappa_2$ implies that there is a small enough constant $\la>0$
such that
\begin{equation}  \notag
    \frac{d}{dt}\CE_{N,m+1}(t)+\la \CD_{N,m+1}(t)\leq C\sqrt{\CE_{N,m}(t)}\CD_{N,m+1}(t).
\end{equation}
Recall that $\CE_{N,m}(t)$ is sufficiently small uniformly in time and $\CE_{N,m+1}(0)$ is finite. Then, it follows that for any $t\geq 0$, $\CE_{N,m+1}(t)$ is finite and satisfies \eqref{app.vw.p03}. Hence, the claim is true for all $m\geq 0$.
We have shown \eqref{app.vw.p01} and  \eqref{thm.energy.1}.   \qed

\subsection{High-order energy estimates}

In this subsection, we consider the proof of the second part of Theorem \ref{thm.energy}. The goal is to construct a high-order instant energy functional $\CE_{N}^{\rm h}(t)$ satisfying the energy inequality \eqref{thm.energy.2} if $\CE_{N}(0)$ is sufficiently small. For that, we suppose that $\CE_{N}(0)$ is sufficiently small through this subsection. Due to \eqref{prop.exi.3}, $\CE_{N}(t)$ is also sufficiently small uniformly in time. Recall also the definition \eqref{def.dNm} of $\CD_N(t)$.

\medskip

\noindent{\it Step 1.} From the system \eqref{VMB.pB}-\eqref{VMB.pM3}, the usual energy as in \cite{S.VMB} gives
\begin{eqnarray}
&\dis \frac{1}{2}\frac{d}{dt}\sum_{1\leq |\al|\leq N}(\|\pa^\al u\|^2+\|\pa^\al [E,B]\|^2)+ \la \sum_{1\leq |\al|\leq N}
\|\nu^{\frac{1}{2}}\pa^\al\{\FI-\FP\} u\|^2\label{app.ho.p01}\\
&\dis \leq C \left(\CE_{N}(t)+\sqrt{\CE_{N}(t)}\right)
\CD_N(t).
\notag
\end{eqnarray}
Multiply equation \eqref{app.vw.p04} by $\{\FI-\FP\}u$,  {integrate it in $x,\xi$ and then use \eqref{VMB.pM1}} to additionally obtain
\begin{multline}\label{app.ho.p02}
\frac{1}{2}\frac{d}{dt}(\|\{\FI-\FP\}u\|^2 {+\|E\|^2})+\la \|\nu^{1/2}\{\FI-\FP\} u\|^2\\
\leq   {\int_{\R^3}E\cdot \na_x\times B dx} +C\|\nu^{1/2}\na_x\FP u\|^2+C\sqrt{\CE_{N}(t)}\CD_N(t).
\end{multline}
 {Furthermore, similar to \eqref{app.vw.p08} from the energy estimate on \eqref{app.vw.p04}, one has}
\begin{multline}
\frac{1}{2}\frac{d}{dt}\sum_{\substack{|\al|+|\be|\leq N\\ |\be|\geq 1}}C_{\al,\be}\|\pa^\al_\be \{\FI-\FP\} u\|^2+\la \sum_{\substack{|\al|+|\be|\leq N\\ |\be|\geq 1}}\|\nu^{\frac{1}{2}}\pa^\al_\be \{\FI-\FP\} u\|^2\label{app.ho.p03}\\
\leq C\sum_{|\al|\leq N}\|\nu^{\frac{1}{2}}\pa^\al \{\FI-\FP\} u\|^2+ C \sum_{|\al|\leq N-1}(\|\pa^\al \na_x\FP u\|^2+\|\pa^\al E\|^2)\\
+ C (\CE_{N}(t)+\sqrt{\CE_{N}(t)})\CD_N(t),
\end{multline}
where $C_{\al,\be}$ are some positive constants.  {For completeness, we give the proof of the above two energy inequalities as follows.}

\medskip

\noindent{\it Proof of \eqref{app.ho.p02} and \eqref{app.ho.p03}:} We first prove \eqref{app.ho.p02}. In fact, \eqref{app.vw.p04} implies
\begin{multline}\label{app.ho.p02.p1}
\frac{1}{2}\frac{d}{dt}\|\{\FI-\FP\}u\|^2 + \langle -\FL\{\FI-\FP\} u,\{\FI-\FP\}u\rangle\\
=\langle \{\FI-\FP\}(E\cdot \xi \FM^{1/2}q_1),\{\FI-\FP\} u\rangle
+\langle -\{\FI-\FP\}(\xi \cdot \na_x \FP u),  \{\FI-\FP\} u\rangle\\
+\langle \Ga(u,u)+\frac{q}{2}E\cdot \xi \{\FI-\FP\}u,\{\FI-\FP\} u\rangle\\
+\langle-\{\FI-\FP\}(q(E+\xi\times B)\cdot \na_\xi \FP u-\frac{q}{2}E\cdot \xi \FP u),\{\FI-\FP\} u\rangle,
\end{multline}
where both the third and fourth terms on the right-hand side are bounded by $C\sqrt{\CE_{N}(t)}\CD_N(t)$, and the  {second term on the right} is bounded by
\begin{equation*}
  \eta\|\nu^{1/2}\{\FI-\FP\} u\|^2+\frac{C}{\eta}\|\nu^{1/2}\na_x\FP u\|^2,
\end{equation*}
for {any small constant $\eta>0$}. To estimate the {first term on the right-hand side}, we use  equation \eqref{VMB.pM1} of $E$ so as to re-write it by
\begin{multline*}
\langle \{\FI-\FP\}(E\cdot \xi \FM^{1/2}q_1),\{\FI-\FP\} u\rangle = \int_{\R^3} E\cdot \langle[\xi,-\xi]\FM^{1/2},\{\FI-\FP\} u\rangle dx\\
= \int_{\R^3} E\cdot (-\pa_t E +\na_x\times B)dx=-\frac{1}{2}\frac{d}{dt}\|E\|^2+\int_{\R^3}E\cdot \na_x\times B dx.
\end{multline*}
Therefore, \eqref{app.ho.p02} follows by plugging the above estimates into \eqref{app.ho.p02.p1}. We now turn to the proof of  \eqref{app.ho.p03}. As in the proof of \eqref{app.vw.p08}, take $\al,\be$ with $|\al|+|\be|\leq N$ and $|\be|\geq 1$. Applying $\pa^\al_\be$ to \eqref{app.vw.p04} and integrating it in $x,\xi$ gives
\begin{equation}\label{app.ho.p03.p1}
    \frac{1}{2}\frac{d}{dt}\|\pa^\al_\be\{\FI-\FP\}u\|^2+\|\nu^{1/2}\pa^\al_\be\{\FI-\FP\}u\|^2=\sum_{i=1}^3\langle I_i,\pa^\al_\be\{\FI-\FP\}u\rangle,
\end{equation}
where $I_i$, $i=1,2,3$, are the same as those three terms on the right-hand side of \eqref{app.vw.p07}. The right-hand side terms of \eqref{app.ho.p03.p1} can be estimated as follows:
\begin{multline*}
  \langle I_1,\pa^\al_\be\{\FI-\FP\}u\rangle \leq  \eta\sum_{|\be'|=|\be|}\| {\nu^{1/2}} \pa^\al_{\be'}\{\FI-\FP\}u\|^2\\
  +C_\eta (\|\pa^\al\{\FI-\FP\}u\|^2+\|\pa^\al E\|^2)
+C\sqrt{\CE_N(t)}\CD_N(t),
\end{multline*}
\begin{multline*}
  \langle I_2,\pa^\al_\be\{\FI-\FP\}u\rangle \leq  \eta\|\pa^\al_{\be}\{\FI-\FP\}u\|^2\\
  +C_\eta (\|{\nu^{1/2}}\pa^\al \na_x\FP u\|^2+\|\pa^\al\na_x\{\FI-\FP\}u\|^2)
+C\sqrt{\CE_N(t)}\CD_N(t),
\end{multline*}
and
\begin{multline*}
  \langle I_3,\pa^\al_\be\{\FI-\FP\}u\rangle \leq  \eta\|\pa^\al_{\be}\{\FI-\FP\}u\|^2\\
  +C_\eta\sum_{\substack{|\al'|+|\be'|\leq N\\ |\be'|\leq  |\be|-1}}\|\nu^{1/2}\pa^{\al'}_{\be'} \{\FI-\FP\} u\|^2
+C\sqrt{\CE_N(t)}\CD_N(t),
\end{multline*}
for a constant $\eta>0$ small enough. Therefore, \eqref{app.ho.p03} follows from multiplying  \eqref{app.ho.p03.p1} by properly chosen large constants $C_{\al,\be}>0$ and taking summation over $\{|\al|+|\be|\leq N,|\be|\geq 1\}$. This completes the proof of \eqref{app.ho.p02} and \eqref{app.ho.p03}. \qed
\medskip

\noindent{\it Step 2.} As in the linearized case, for simplicity, we still denote
\begin{equation}
\notag
   G=\langle [\xi,-\xi]\FM^{1/2}, \{\FI-\FP\} u\rangle=\langle \xi\FM^{1/2}, \{\FI-\FP\} u\cdot q_1\rangle.
\end{equation}
Corresponding to \eqref{macro.1} and \eqref{macro.2} for the linearized version, one has from the balance laws \eqref{m0.non}-\eqref{m2.non}
and the high-order moment equations \eqref{m2ii.non}-\eqref{m3.non} in the nonlinear case that
\begin{equation}\label{macro.non1}
    \left\{
    \begin{array}{l}
      \dis \pa_t\left(\frac{a_++a_-}{2}\right)+\na_x\cdot b=0,\\
      \dis \pa_t b_i+\pa_i\left(\frac{a_++a_-}{2}+2c\right)+\frac{1}{2}\sum_{j=1}^3\pa_j\highG_{ij}(\{\FI-\FP\}u\cdot [1,1])\\
      \dis \hspace{3.5cm}
      = {E_i\frac{a_+-a_-}{2} +[G\times B]_i},\\
      \dis \pa_t c+ \frac{1}{3}\na_x\cdot b +\frac{5}{6}\sum_{i=1}^3\partial_i \highB_i(\{\FI-\FP\}u\cdot [1,1])=\frac{1}{6}G\cdot E,
    \end{array}\right.
\end{equation}
and
\begin{equation}\label{macro.non2}
    \left\{
    \begin{array}{l}
      \dis \pa_t \left[\frac{1}{2}\highG_{ij}(\{\FI-\FP\}u\cdot [1,1]) +2c\de_{ij}\right]\\
      \dis \hspace{1.5cm}+\pa_i b_j+\pa_j b_i
      =
      \frac{1}{2}\highG_{ij}((l_++l_-)+(\sourceG_++\sourceG_-)),\\
      \dis \frac{1}{2}\pa_t \highB_i(\{\FI-\FP\}u\cdot [1,1])+\pa_i c=\frac{1}{2}\highB_i((l_++l_-)+(\sourceG_++\sourceG_-)),
    \end{array}\right.
\end{equation}
for $ 1\leq i,j\leq 3$, where $l_\pm$ is still defined in \eqref{def.l} and $g$ is defined in \eqref{def.g.non}.  Furthermore, by using the third equation of \eqref{macro.non1} to replace $\pa_tc$ in the first equation of  \eqref{macro.non2}, one has
\begin{multline}\label{macro.non.bij}
\frac{1}{2}\pa_t\highG_{ij}(\{\FI-\FP\}u\cdot [1,1])+\pa_i b_j+\pa_j b_i-\frac{2}{3}\de_{ij}\na_x \cdot b \\
-\frac{5}{3}\de_{ij}\na_x\cdot \highB(\{\FI-\FP\}u\cdot [1,1])\\
      =\frac{1}{2}\highG_{ij}((l_++l_-)+(\sourceG_++\sourceG_-)) {-\frac{1}{3}\de_{ij}G\cdot E}.
\end{multline}
Similarly, corresponding to \eqref{m0-} and \eqref{m1-},  it follows from the balance laws \eqref{m0.non} and \eqref{m1.non} that
\begin{eqnarray}
&&\pa_t (a_+-a_-)+\na_x\cdot G=0,\label{m0-.non}\\
\notag
&&\pa_t G + \na_x (a_+-a_-)-2E+ \na_x\cdot \highG (\{\FI-\FP\}u\cdot q_1)\\
&&\qquad\qquad=E(a_++a_-)+2b\times B {+\langle [\xi,-\xi]\FM^{1/2},\FL u +\Ga(u,u)\rangle}.\label{m1-.non}
\end{eqnarray}
We also recall
\begin{equation}\label{m.div.e}
    \na_x\cdot E=a_+-a_-.
\end{equation}


\begin{lemma}\label{lem.non.ma.diss}
One has the following four estimates
\begin{equation}\label{lem.non.ma.diss.1}
    \sum_{|\al|\leq N-1}\|\pa^\al \na_x [b,c]\|^2\leq C\sum_{|\al|\leq N}\|\pa^\al \{\FI-\FP\}u\|^2+C\CE_{N}(t)\CD_N(t),
\end{equation}
\begin{multline}
\label{lem.non.ma.diss.3}
\frac{d}{dt}\sum_{1\leq |\al|\leq N-1}\int_{\R^3}\pa^\al b\cdot \pa^\al \na_x (a_++a_-)dx
\\
+\la \sum_{1\leq |\al|\leq N-1}\|\pa^\al \na_x (a_++a_-)\|^2
\\
\leq C \sum_{|\al|\leq N}\|\pa^\al \{\FI-\FP\}u\|^2+C\CE_{N}(t)\CD_N(t),
\end{multline}
\begin{multline}
\label{lem.non.ma.diss.2}
\frac{d}{dt}\sum_{|\al|\leq N-1}\int_{\R^3}\pa^\al G\cdot \pa^\al \na_x (a_+-a_-)dx\\
+\la \sum_{|\al|\leq N-1}(\|\pa^\al \na_x (a_+-a_-)\|^2+\|\pa^\al(a_+-a_-)\|^2)\\
\leq C \sum_{|\al|\leq N}\|\pa^\al \{\FI-\FP\}u\|^2+C\CE_{N}(t)\CD_N(t),
\end{multline}
and
\begin{multline}
\label{lem.non.ma.diss.4}
 \sum_{|\al|\leq N-1}\|\pa^\al \pa_t[a_+\pm a_-,b,c]\|^2\leq C\sum_{|\al|\leq N-1}\|\pa^\al \na_x (a_+ + a_-)\|^2\\
 +C \sum_{|\al|\leq N}\|\pa^\al \{\FI-\FP\}u\|^2+C\CE_{N}(t)\CD_N(t),
\end{multline}
for any $t\geq 0$.
\end{lemma}

\begin{proof}
First consider \eqref{lem.non.ma.diss.1}. For $\al$ with $|\al|\leq N-1$, the bound of $c$ follows from
\begin{eqnarray*}
\|\pa^\al \pa_i c\|^2&\leq& C\|\pa^\al \pa_t \highB_i(\{\FI-\FP\}u\cdot [1,1])\|^2+C\|\pa^\al \highB_i((l_++l_-)+(\sourceG_++\sourceG_-))\|\\
&\leq& C\|\pa^\al\pa_t\{\FI-\FP\}u\|^2+C\|\pa^\al\na_x\{\FI-\FP\}u\|^2+C\|\pa^\al\{\FI-\FP\}u\|^2\\
&&
+C\|\pa^\al\highB_i(g_\pm)\|^2\\
&\leq &C\sum_{|\al|\leq N}\|\pa^\al\{\FI-\FP\}u\|^2+C\CE_N(t)\CD_N(t)
\end{eqnarray*}
due to the second equation of \eqref{macro.non2}, and in the same way, the bound of $b$ follows from equation \eqref{macro.non.bij} by noticing the identity
\begin{equation}  \notag
    \int_{\R^3}|\pa^\al (\pa_i b_j+\pa_j b_i-\frac{2}{3}\de_{ij}\na_x \cdot b)|^2dx=2\|\pa^\al \na_x b\|^2+\frac{2}{3}\|\pa^\al \na_x\cdot b\|^2.
\end{equation}
Next, \eqref{lem.non.ma.diss.3} follows from the first two equations of \eqref{macro.non1}. In fact, for $\al$ with $1\leq |\al|\leq N-1$, the second equation of \eqref{macro.non1} implies
\begin{multline*}
  \frac{1}{2}\|\pa^\al \pa_i(a_+ + a_-)\|^2 \\
  =-\frac{d}{dt}\int_{\R^3}\pa^\al \pa_i (a_+ + a_-)\pa^\al b_i dx
  +2\int_{\R^3}\pa^\al \pa_i\pa_t\frac{a_++a_-}{2}\pa^\al b_i dx\\
\quad\quad\quad+\int_{\R^3} \pa^\al \pa_i (a_+ + a_-) \pa^\al \Big\{-2\pa_ic-\frac{1}{2}\sum_{j=1}^3\pa_j\highG_{ij}(\{\FI-\FP\}u\cdot [1,1])   \\
  +E_i\frac{a_+-a_-}{2} + {[G\times B]_i} \Big\} dx,
\end{multline*}
and further using the first equation of \eqref{macro.non1} and Cauchy-Schwarz gives
\begin{eqnarray*}
&&\frac{d}{dt}\sum_{1\leq |\al|\leq N-1}\int_{\R^3}\pa^\al b\cdot \pa^\al \na_x (a_++a_-)dx
+\la \sum_{1\leq |\al|\leq N-1}\|\pa^\al \na_x (a_++a_-)\|^2
\\
&&\leq C\sum_{1\leq |\al|\leq N-1}\|\pa^\al \na_x[b,c]\|^2+ C \sum_{1\leq |\al|\leq N-1}\|\pa^\al \na_x\{\FI-\FP\}u\|^2+C\CE_{N}(t)\CD_N(t)\\
&&\leq C \sum_{|\al|\leq N}\|\pa^\al \{\FI-\FP\}u\|^2+C\CE_{N}(t)\CD_N(t),
\end{eqnarray*}
where \eqref{lem.non.ma.diss.1} was used in the last inequality. Hence, \eqref{lem.non.ma.diss.3}  holds. To prove \eqref{lem.non.ma.diss.2}, for $\al$ with $|\al|\leq N-1$, \eqref{m1-.non} together with \eqref{m0-.non} and \eqref{m.div.e} yield
\begin{eqnarray*}
&&\|\pa^\al \na_x(a_+ -a_-)\|^2 + 2\|\pa^\al (a_+ - a_-)\|^2\\
&&=\int_{\R^3} \pa^\al \na_x(a_+ -a_-)\cdot \pa^\al [\na_x(a_+ -a_-)-2E]dx\\
&&=-\frac{d}{dt}\int_{\R^3}\pa^\al G\cdot \pa^\al \na_x (a_+-a_-)dx+\|\pa^\al \na_x\cdot G\|^2\\
&&\quad+\int_{\R^3} \pa^\al \na_x(a_+ -a_-)\cdot\pa^\al \big[-\na_x\cdot \highG (\{\FI-\FP\}u\cdot q_1)+E(a_++a_-)\\
&&\quad\quad\quad\quad\quad\quad+2b\times B+\langle [\xi,-\xi]\FM^{1/2},\FL \{\FI-\FP\}u +\Ga(u,u)\rangle \big]dx,
\end{eqnarray*}
which implies \eqref{lem.non.ma.diss.2} after using Cauchy-Schwarz and taking summation over $|\al|\leq N-1$. Finally, it is straightforward to verify \eqref{lem.non.ma.diss.4} from \eqref{macro.non1} and \eqref{m0-.non} as well as \eqref{lem.non.ma.diss.1}. The proof of Lemma \ref{lem.non.ma.diss} is complete.
\end{proof}

Furthermore, \eqref{app.p2} and \eqref{app.p3} for the upper bounds of $E,B$ can be refined as

\begin{lemma}\label{lem.bdd.EB}
It holds that
\begin{multline}
\label{lem.bdd.EB.1}
\sum_{|\al|\leq N-1}\|\pa^\al E\|^2+\sum_{1\leq |\al|\leq N-1}\|\pa^\al B\|^2\\
\leq C\sum_{|\al|\leq N-1}\|\pa^\al \na_x\FP u\|^2
+C\sum_{|\al|\leq N}\|\pa^\al \{\FI-\FP\}u\|^2+C\CE_{N}(t)\CD_N(t),
\end{multline}
for any $t\geq 0$.
\end{lemma}

\begin{proof}
One again, we use \eqref{m1-.non}, that is,
\begin{multline}
2E=\pa_t G + \na_x (a_+-a_-)+ \na_x\cdot \highG (\{\FI-\FP\}u\cdot q_1)\\
-E(a_++a_-)-2b\times B-\langle [\xi,-\xi]\FM^{1/2},\FL  \{\FI-\FP\}u +\Ga(u,u)\rangle.
\end{multline}
Then, the upper bound for $E$ in \eqref{lem.bdd.EB.1} follows directly from the above equation, and thus the upper bound of $B$ also holds by \eqref{app.p3} and  the estimate on $E$.
\end{proof}

\medskip

\noindent{\it Step 3.} Let us define the interactive high-order instant energy functional $ \CE_{N,{\rm int}}^{{\rm h}}(t)$ as in \cite{DY-09VPB} by
\begin{multline}\label{def.int.tf}
  \CE_{N,{\rm int}}^{{\rm h}}(t) = \sum_{1\leq |\al|\leq N-1}\int_{\R^3}\pa^\al b\cdot \pa^\al \na_x (a_++a_-)dx\\
  +\sum_{|\al|\leq N-1}\int_{\R^3}\pa^\al G\cdot \pa^\al \na_x (a_+-a_-)dx.
\end{multline}
Notice that
\begin{equation}\label{he.int.bdd}
    | \CE_{N,{\rm int}}^{{\rm h}}(t)|\leq C\sum_{1\leq |\al|\leq N}\|\pa^\al \FP u\|^2+C\sum_{|\al|\leq N-1}\| {\pa^\al}\{\FI-\FP\}u\|^2.
\end{equation}
In addition,  {the proper linear combination of} \eqref{lem.non.ma.diss.1}, \eqref{lem.non.ma.diss.3},  \eqref{lem.non.ma.diss.2} and \eqref{lem.non.ma.diss.4} in Lemma \ref{lem.non.ma.diss} implies
\begin{multline}
 \nonumber
\frac{d}{dt}\CE_{N,{\rm int}}^{{\rm h}}(t) {+\la \sum_{1\leq |\al|\leq N}\|\pa^\al [a_+\pm a_-,b,c]\|^2+\la \|a_+-a_-\|^2}\\
\leq C\|\na_x (a_+ + a_-)\|^2+ C\sum_{|\al|\leq N}\|\pa^\al \{\FI-\FP\}u\|^2+C\CE_{N}(t)\CD_N(t).
\end{multline}
Note that to derive the above inequality  {we can add}
$\|\nabla_x(a_++a_-)\|^2$ to both sides.
After further plugging in \eqref{lem.bdd.EB.1}, one has
\begin{multline}
 \nonumber
\frac{d}{dt}\CE_{N,{\rm int}}^{{\rm h}}(t) {+\la \sum_{1\leq |\al|\leq N}\|\pa^\al [a_+\pm a_-,b,c]\|^2+\la \|a_+-a_-\|^2}\\
+\la \sum_{|\al|\leq N-1}\|\pa^\al E\|^2+\la \sum_{1\leq |\al|\leq N-1}\|\pa^\al B\|^2\\
\leq C\|\na_x\FP u\|^2+ C\sum_{|\al|\leq N}\|\pa^\al \{\FI-\FP\}u\|^2+C\CE_{N}(t)\CD_N(t).
\end{multline}
Hence, it further holds that
\begin{multline}\label{macro.total.diss}
\frac{d}{dt}\CE_{N,{\rm int}}^{{\rm h}}(t) {+\la \sum_{1\leq |\al|\leq N}\|\pa^\al \FP u\|^2}
+\la \sum_{1\leq |\al|\leq N-1}\|\pa^\al [E,B]\|^2+\la \|E\|^2\\
\leq C\|\na_x\FP u\|^2+ C\sum_{|\al|\leq N}\|\pa^\al \{\FI-\FP\}u\|^2+C\CE_{N}(t)\CD_N(t).
\end{multline}
This is the main energy estimate for $\CE_{N,{\rm int}}^{{\rm h}}(t)$.

\medskip

\noindent{\it Step 4.} Notice that \eqref{app.ho.p02} implies
\begin{multline}\label{app.ho.p02.re}
     \frac{1}{2}\frac{d}{dt}(\|\{\FI-\FP\}u\|^2 {+\|E\|^2})+\la \|\nu^{\frac{1}{2}}\{\FI-\FP\}u\|^2\\
     \leq  {\eta \|E\|^2 +\frac{1}{4\eta}\|\na_x\times B\|^2+
     C \|
     \nu^{1/2}
     \na_x\FP u\|^2}\\
    + C\left(\CE_{N}(t)+\sqrt{\CE_{N}(t)}\right)\CD_N(t)
\end{multline}
for a constant $\eta>0$ to be chosen small enough.
Now, we are ready to construct $\CE^{\rm h}_{N}(t)$. In fact, let us define
\begin{eqnarray*}
 \CE^{\rm h}_{N}(t) &=& \sum_{1\leq |\al|\leq N}(\|\pa^\al u\|^2+\|\pa^\al [E,B]\|^2)+\kappa_1(\|\{\FI-\FP\}u\|^2 {+\|E\|^2})\\
 &&+\kappa_2\CE_{N,{\rm int}}^{{\rm h}}(t)
 +\kappa_3\sum_{\substack{|\al|+|\be|\leq N\\ |\be|\geq 1}}C_{\al,\be}\|\pa^\al_\be \{\FI-\FP\} u\|^2,
\end{eqnarray*}
for suitable constants $0<\kappa_3\ll \kappa_2\ll \kappa_1\ll 1$ to
be determined now, where $\CE_{N,{\rm int}}^{{\rm h}}(t)$ is given
by \eqref{def.int.tf}. Due to \eqref{he.int.bdd}, one can let
$0<\kappa_3\ll \kappa_2\ll \kappa_1\ll 1$ be small enough such that
\eqref{def.eNm.h} holds with $m=0$ and hence $\CE^{\rm h}_{N}(t)$ is
indeed a well-defined high-order instant energy functional. In
addition, by choosing $0<\kappa_3\ll \kappa_2\ll \kappa_1\ll 1$
further small enough, the sum of \eqref{app.ho.p01},
\eqref{app.ho.p02.re}$\times \kappa_1$ for $\eta>0$ small enough,
\eqref{macro.total.diss}$\times \kappa_2$ and
\eqref{app.ho.p03}$\times \kappa_3$ yields
\begin{multline}  \notag
 \frac{d}{dt}\CE_N^{\rm h}(t)+\la \CD_N(t)\leq
 C {(\|\nu ^{1/2}\na_x\FP u\|^2+\|\na_x\times B\|^2)}\\
 +C\left(\sqrt{\CE_{N}(t)}+ \CE_{N}(t) \right)
 \CD_N(t),
\end{multline}
which implies the desired estimate \eqref{thm.energy.2} since $\CE_{N}(t)$ is small enough uniformly in all $t\geq 0$. This completes the proof of the second part of Theorem \ref{thm.energy}. \qed

\section{Time decay for the nonlinear system}\label{sec.decayNL}

This section is devoted to the proof of Theorem \ref{thm.ns} and
Corollary \ref{cor}. As mentioned at the beginning of Section
\ref{sec.energy}, the time-decay rates of solutions in some energy
spaces are obtained by combining the corresponding nonlinear energy
estimates and the linearized time-decay property applied to the
lower-order terms. In addition, we also will use the time-weighted
estimates, as in \eqref{def.einfty}, and a bootstrap argument on the derivatives to handle the
regularity-loss phenomenon in the energy dissipation rate due to the
degenerately dissipative (regularity loss) property from the Maxwell system. Decay
rates of solutions in the $L^\infty$-norm are deduced by using some
optimal Sobolev embedding theorem, and hence time rates for
$L^r$-norm with $2\leq r\leq \infty$ will follow from the normal
$L^2$-$L^\infty$ interpolation inequality.

\subsection{Decay rates of the full instant energy functional}\label{sec.decayNL.1}

In what follows, we fix an integer $m\geq 0$ and suppose that $\eps_{N+2,m\vee 1}$ defined in \eqref{def.id.jm} from initial data $[u_0,E_0,B_0]$ is sufficiently small. The goal is to prove \eqref{thm.ns.2}. Since $\CE_{N,m}(0)\leq \eps_{N+2,m\vee 1}$ is sufficiently small, the Lyapunov inequality \eqref{thm.energy.1} for  $\CE_{N,m}(t)$ and  $\CD_{N,m}(t)$ holds for any $t\geq 0$. Let $1<\ell<2$. Multiplying \eqref{thm.energy.1} by $(1+t)^\ell$ and then taking time integration over $[0,t]$ gives
\begin{equation}  \notag
    (1+t)^\ell\CE_{N,m}(t)+\la \int_0^t (1+s)^\ell \CD_{N,m}(s)ds\leq \CE_{N,m}(0)+\ell\int_0^t {(1+s)}^{\ell-1}\CE_{N,m}(s)ds.
\end{equation}
Recall \eqref{def.eNm} and \eqref{def.dNm} for definitions of $\CE_{N,m}(t)$ and  $\CD_{N,m}(t)$. It follows from \eqref{def.eNm} and \eqref{def.dNm} that
\begin{equation}  \notag
    \CE_{N,m}(t)\leq C(\|\FP u {(t)}\|^2+\|B {(t)}\|^2+\CD_{N+1,m}(t)).
\end{equation}
Note that we could take $\CD_{N+1,m-1}(t)$ above.
Then, it follows that
\begin{multline*}
    (1+t)^\ell\CE_{N,m}(t)+ \la \int_0^t (1+s)^\ell \CD_{N,m}(s)ds\\
    \leq \,\CE_{N,m}(0)+C\ell\int_0^t(1+s)^{\ell-1}\left(\|\FP u(s)\|^2+\|B(s)\|^2\right)ds\\
    +C\ell\int_0^t(1+s)^{\ell-1}\CD_{N+1,m}(s)ds.
\end{multline*}
We use similar estimates from \eqref{thm.energy.1} for $N+1$ and $N+2$ as above to obtain
\begin{multline*}
    (1+t)^{\ell-1}\CE_{N+1,m}(t)+ \la \int_0^t (1+s)^{\ell-1} \CD_{N+1,m}(s)ds\\
    \leq \,\CE_{N+1,m}(0)+C(\ell-1)\int_0^t(1+s)^{\ell-2}(\|\FP u(s)\|^2+\|B(s)\|^2)ds\nonumber\\
    +C(\ell-1)\int_0^t(1+s)^{\ell-2}\CD_{N+2,m}(s)ds,
\end{multline*}
and
\begin{equation}  \notag
    \CE_{N+2,m}(t)+\la \int_0^t\CD_{N+2,m}(s)ds\leq \CE_{N+2,m}(0).
\end{equation}
It follows by iteration that
\begin{multline}
    (1+t)^\ell \CE_{N,m}(t)+ \la \int_0^t (1+s)^\ell \CD_{N,m}(s)ds\\
    \leq C\CE_{N+2,m}(0)+C\ell\int_0^t(1+s)^{\ell-1}(\|\FP u(s)\|^2+\|B(s)\|^2)ds,
    \label{thm.ns.p1}
\end{multline}
due to $1<\ell<2$.

On the other hand, define
\begin{equation}\label{def.einfty}
    \CE_{N,m}^\infty(t) \eqdef \sup_{0\leq s\leq t}(1+s)^{\frac{3}{2}}\CE_{N,m}(s).
\end{equation}
From now on let us consider the
pointwise time-decay estimate on $\|\FP u\|^2+\|B\|^2$ in terms of $ \CE_{N,m}^\infty(t)$. Formally, the solution $U=[u,E,B]$ to the Cauchy problem \eqref{VMB.pB}, \eqref{VMB.pM1}, \eqref{VMB.pM2},  \eqref{VMB.pM3} of the Vlasov-Maxwell-Boltzmann system can be written as
\begin{equation}\label{decom.U}
   U(t)=\semiG(t)U_0+\int_0^t\semiG(t-s)[\sourceG(s),0,0]ds,
\end{equation}
where
$\semiG$ is defined in \eqref{def.lu1} and
$\sourceG=[\sourceG_+, \sourceG_-]$ is given by \eqref{def.g.non}; equivalently
\begin{equation}  \notag
    \sourceG \eqdef \Ga(u,u)+\frac{q}{2}E\cdot \xi u-q (E+\xi\times B)\cdot \na_\xi u.
\end{equation}
For later use, $U(t)$ can be rewritten as
\begin{equation}  \notag
   U(t)=I_0(t)+I_1(t)+I_2(t),
\end{equation}
with
\begin{eqnarray*}
  I_0(t) &=&  \semiG(t)U_0,\\
  I_1(t) &=&\int_0^t\semiG(t-s)[\sourceG_1(s),0,0]ds,\\
  I_2(t) &=&\int_0^t\semiG(t-s)[\sourceG_2(s),0,0]ds,
\end{eqnarray*}
where
\begin{eqnarray*}
  \sourceG_1 &=& \Ga(u,u)+\{\FI-\FP\}\left(\frac{q}{2}E\cdot \xi u-q (E+\xi\times B)\cdot \na_\xi u\right),\\
    \sourceG_2 &=& \FP \left( \frac{q}{2}E\cdot \xi u-q (E+\xi\times B)\cdot \na_\xi u\right).
\end{eqnarray*}
Here, we remark that time-integral terms $I_1(t)$, $I_2(t)$ and the time-integral term in $U(t)$ are well-defined since
\begin{equation}
    \int_{\R^3}\FM^{1/2}[\sourceG_{i,+}(s)-\sourceG_{i,-}(s)]d\xi=0,\ \ i=1,2,
\end{equation}
for all $x\in \R^3$. For later use,  we will need

\begin{lemma}\label{lem.h1h2}
Assume $N\geq 4$. For any integer $j\geq 0$, there is $C>0$ such
that
\begin{eqnarray}
&\dis \|\nu^{-1/2}\sourceG_1(t)\|_{L^2_\xi(H^j_x)\cap Z_1}^2\leq
C\CE_{{j\vee N},1}(t)\CE_{j\vee N}(t),\label{lem.h1h2.1}
\\
&\dis \|\sourceG_2(t)\|_{L^2_\xi(H^j_x)\cap Z_1}\leq C\CE_{j\vee
N}(t),\label{lem.h1h2.2}
\end{eqnarray}
for any $t\geq 0$, where $j\vee N\eqdef\max\{j,N\}$.
\end{lemma}

\begin{proof}
Due to the definition of the quadratically nonlinear function $g_2$
and the fact that $\FP$ is a macroscopic projector,
\eqref{lem.h1h2.2} immediately follows by the definition
\eqref{def.eNm} for $\CE_N(t)=\CE_{N,0}(t)$ with $N\geq 4$. To prove
\eqref{lem.h1h2.1},
\begin{multline*}
 \|\nu^{-1/2}\sourceG_1(t)\|_{L^2_\xi(H^j_x)\cap Z_1}^2\leq
2\|\nu^{-1/2}\Ga(u,u)\|_{L^2_\xi(H^j_x)\cap
Z_1}^2\\
 +2\|\nu^{-1/2}\left(\frac{q}{2}E\cdot \xi u-q
(E+\xi\times B)\cdot \na_\xi u\right)\|_{L^2_\xi(H^j_x)\cap Z_1}^2\\
 +2\|\nu^{-1/2}\FP\left(\frac{q}{2}E\cdot \xi u-q
(E+\xi\times B)\cdot \na_\xi u\right)\|_{L^2_\xi(H^j_x)\cap Z_1}^2,
\end{multline*}
where it is also straightforward to see that the last two terms on
the r.h.s.~are bounded by $C\CE_{{j\vee N},1}(t)\CE_{j\vee
N}(t)$, while the estimate on the first term follows from
\cite[Lemma 3.3]{DUYZ}. Hence, \eqref{lem.h1h2.1} and
\eqref{lem.h1h2.2} are proved.
\end{proof}

We recall the norms from \eqref{brief.norm}
for $U=[u,E,B]$.
We now apply \eqref{thm.ls.1} with $m=0$, $r=1$ and $\ell=3/2$ to $I_0(t)$ and $I_2(t)$, respectively, to obtain
\begin{equation}  \notag
    \|I_0(t)\|_{\CL^2}\leq C(1+t)^{-\frac{3}{4}} \|U_0\|_{\CH^{\frac{3}{2}}\cap \CZ_1},
\end{equation}
and
\begin{multline*}
    \|I_2(t)\|_{\CL^2}\leq \int_0^t\|\semiG(t-s)[\sourceG_2(s),0,0]\|_{\CL^2}ds
    \\
    \leq
    C\int_0^t(1+t-s)^{-\frac{3}{4}}\|\sourceG_2(s)\|_{L^2_\xi(H^{\frac{3}{2}}_x)\cap Z_1}ds.
\end{multline*}
For $I_2(t)$, it further holds from Lemma \ref{lem.h1h2} and the definition \eqref{def.einfty} of $\CE_{N,m}^\infty(t)$ that
\begin{multline*}
    \|I_2(t)\|_{\CL^2}\leq C\int_0^t(1+t-s)^{-\frac{3}{4}}\CE_{N}(s)ds\leq C\int_0^t(1+t-s)^{-\frac{3}{4}}\CE_{N,m}(s)ds\\
    \leq C \CE_{N,m}^\infty(t)\int_0^t(1+t-s)^{-\frac{3}{4}}(1+s)^{-\frac{3}{2}}ds\\
    \leq C \CE_{N,m}^\infty(t)(1+t)^{-\frac{3}{4}}.
\end{multline*}
Similarly, applying \eqref{thm.ls.2} with $m=0$, $r=1$ and $\ell=3/2$ to $I_1(t)$ due to the fact that $\FP \sourceG_1=0$ and then using Lemma  \ref{lem.h1h2}, one has
\begin{multline*}
  \|I_1(t)\|_{\CL^2}^2 \leq  C\int_0^t(1+t-s)^{-\frac{3}{2}} \|\nu^{-1/2}\sourceG_1(s)\|_{L^2_\xi(H^{\frac{3}{2}}_x)\cap Z_1}^2ds\\
  \leq  C\int_0^t(1+t-s)^{-\frac{3}{2}}\CE_{N,1}(s)\CE_N(s)ds\\
   \leq  C\int_0^t(1+t-s)^{-\frac{3}{2}}\CE_{N,m\vee 1}(s)\CE_{N,m}(s)ds.
\end{multline*}
Further using the definition \eqref{def.einfty} of
$\CE_{N,m}^\infty(t)$ again and
 \eqref{thm.energy.1}, it follows that
\begin{multline*}
  \|I_1(t)\|_{\CL^2}^2 \leq  C\CE_{N,m}^\infty(t)
  \sup_{\tau\geq 0}\CE_{N,m\vee 1}(\tau)
  \int_0^t(1+t-s)^{-\frac{3}{2}}(1+s)^{-\frac{3}{2}}ds\\
  \leq C (1+t)^{-\frac{3}{2}} \CE_{N,m}^\infty(t)\CE_{N,m\vee 1}(0).
\end{multline*}
Collecting the estimates on $I_i(t)$ $(i=0,1,2)$ above implies
\begin{multline}
\|\FP u(t)\|^2+\|B(t)\|^2\leq C\|U(t)\|_{\CL^2}^2\leq C\sum_{i=0}^2\|I_i(t)\|_{\CL^2}^2\\
\leq  C(1+t)^{-\frac{3}{2}} \|U_0\|_{\CH^2\cap \CZ_1}^2+C \CE_{N,m\vee 1}(0)\CE_{ {N,m}}^\infty(t)(1+t)^{-\frac{3}{2}}\\
+C (1+t)^{-\frac{3}{2}}[\CE_{N,m}^\infty(t)]^2.\label{thm.ns.p2}
\end{multline}
Now, fix a constant $\eps>0$ close to zero. Taking $\ell=3/2+\eps$ in \eqref{thm.ns.p1} yields
\begin{multline*}
    (1+t)^{\frac{3}{2}+\eps}\CE_{N,m}(t)+ \la \int_0^t (1+s)^{\frac{3}{2}+\eps}\CD_{N,m}(s)ds\\
    \leq \,C\CE_{N+2,m}(0)+C\int_0^t(1+s)^{\frac{3}{2}+\eps-1}(\|\FP u(s)\|^2+\|B(s)\|^2)ds.
\end{multline*}
Plugging \eqref{thm.ns.p2} into the above inequality gives
\begin{multline*}
 \dis  (1+t)^{\frac{3}{2}+\eps} \CE_{N,m}(t)\leq  C\CE_{N+2,m}(0)\\
  \dis +C_\eps(1+t)^{\eps}\left(\|U_0\|_{\CH^2\cap \CZ_1}^2+\CE_{N,m\vee1}(0)\CE_{N,m}^\infty(t)+[\CE_{N,m}^\infty(t)]^2\right),
\end{multline*}
which after dividing by $(1+t)^\eps$ implies that for any $t\geq 0$,
\begin{equation}  \notag
\dis  (1+t)^{\frac{3}{2}} \CE_{N,m}(t)\leq  C
\left(\CE_{N+2,m}(0)+\|U_0\|_{\CH^2\cap \CZ_1}^2+\CE_{N,m\vee1}(0)\CE_{N,m}^\infty(t)+[\CE_{N,m}^\infty(t)]^2\right).
\end{equation}
We equivalently have that
\begin{equation}  \notag
\dis  \CE_{N,m}^\infty(t)\leq  C
\left(\CE_{N+2,m}(0)+\|U_0\|_{\CH^2\cap \CZ_1}^2+\CE_{N,m\vee1}(0)\CE_{N,m}^\infty(t)+[\CE_{N,m}^\infty(t)]^2\right).
\end{equation}
Since $\CE_{N+2,m}(0)+\|U_0\|_{\CH^2\cap \CZ_1}^2\leq \eps_{N+2,m\vee 1}$  and $\CE_{N,m\vee1}(0)\leq \eps_{N+2,m\vee 1}$ are sufficiently small, one has
\begin{equation}
\notag
\CE_{N,m}^\infty(t)\leq  C\eps_{N+2,m},
\end{equation}
which gives the desired time-decay estimate \eqref{thm.ns.2}. \qed

\subsection{Decay rates of the high-order instant energy functional}\label{sec.decayNL.2} In this subsection, we shall prove the time-decay estimate \eqref{thm.ns.3} for the high-order instant energy functional $\CE_N^{\rm h}(t)$. It follows from Theorem \ref{thm.energy} that
\begin{equation}\label{dr.he.p01}
    \frac{d}{dt} \CE_N^{\rm h}(t)+\la \CD_{N}(t)\leq
    C(\|
          {
     \nu^{1/2}
     }
    \na_x\FP u(t)\|^2 {+\|\na_x\times B(t)\|^2}),
\end{equation}
which by definitions \eqref{def.eNm.h}, \eqref{def.dNm} of $\CE_N^{\rm h}(t)$ and $ \CD_{N}(t)$,  implies
\begin{multline}\label{dr.he.p02}
    \frac{d}{dt} \CE_N^{\rm h}(t)+\la \CE_N^{\rm h}(t)\leq
    C(\|
          {
     \nu^{1/2}
     }
    \na_x\FP u(t)\|^2 {+\|\na_x\times B(t)\|^2})\\
    +C\sum_{|\al|=N}\|\pa^\al [E(t),B(t)]\|^2.
\end{multline}
Notice that from the system \eqref{VMB.pB}-\eqref{VMB.pM2}, one can replace all time derivatives of $E$ and $B$ by {the spatial derivatives} of $E,B$ and space-time derivatives of $\{\FI-\FP\} u$.

In particular we {\it claim} that
\begin{multline}\label{time.d.claim}
    \sum_{|\al|=N}\|\pa^\al [E(t),B(t)]\|^2\leq C \sum_{\substack{  |\al'|=N-1  }} \| \{\FI-\FP\} \pa^{\al'} u(t)\|^2\\
    +C\sum_{\substack{|\al| = N\\ \al_0=0}}\|\pa^\al [E(t),B(t)]\|^2.
\end{multline}

\medskip
\noindent{\it Proof of claim:} This estimate \eqref{time.d.claim}
follows from an induction on the number of time derivatives using
\eqref{maxwell.u}. First suppose $|\al|=N$ and $\alpha_0=1$.  Then
$\alpha = (\alpha_0, 0, 0, 0) + \alpha'$ with $\alpha' = (0,
\alpha_1, \alpha_2, \alpha_3)$ and $|\al'|=N-1$.   Now using
\eqref{max.B.est} we obtain
\begin{equation}  \notag
    \|\pa^\al B\|^2\leq C\sum_{\substack{ |\al| =1 \\ \al_0=0}}\|\pa^{\al'} \pa^\al E\|^2
    \le
    C\sum_{\substack{1\leq |\al|\leq N\\ \al_0=0}}\|\pa^\al E\|^2.
\end{equation}
Similarly, using \eqref{maxwell.u} and also  {Cauchy-Schwarz}, in this
initial case we have
\begin{equation}  \notag
    \|\pa^\al E\|^2\leq  \|\pa^{\al'} \{\FI-\FP\}u\|^2+
    C\sum_{\substack{ |\al| =1 \\ \al_0=0}}\|\pa^{\al'} \pa^\al B\|^2.
\end{equation}
This completes the first inductive step.

Now we suppose that the lemma is true for $|\al|=N$ and $\alpha_0=m$
with any $m=1,\ldots, N-1$. We will show that the lemma is also true
for $\alpha_0=m+1$. Similar to the first inductive step, as in
\eqref{max.B.est} we have
\begin{equation}\notag
    \pa^\al B=\pa^{\al'}\pa_tB=-\pa^{\al'}\na_x\times E,
\end{equation}
for some $\al'$ with $|\al'|=|\al|-1$ since $\alpha_0=m+1>0$. Since
$\al'_0=m$, we achieve
\begin{equation}  \notag
    \|\pa^\al B\|^2\leq C \sum_{\substack{  |\al'|=N-1  \\ \al'_0=m}} \sum_{\substack{ |\al| =1 \\ \al_0=0}}\|\pa^{\al'} \pa^\al E\|^2.
\end{equation}
Now the general estimate for $B$ follows from the inductive
hypothesis.

For the last term, we use \eqref{maxwell.u}$_1$ to observe that
\begin{equation}\notag
    \pa^\al E=\pa^{\al'}\pa_tE=\pa^{\al'}\na_x\times B - \langle[\xi,-\xi]\FM^{1/2}, \{\FI-\FP\}\pa^{\al'}u\rangle.
\end{equation}
Similarly, using  {Cauchy-Schwarz} we observe that
\begin{equation}  \notag
    \|\pa^\al E\|^2\leq
    C \sum_{\substack{  |\al'|=N-1  \\ \al'_0=m}} \sum_{\substack{ |\al| =1 \\ \al_0=0}}\|\pa^{\al'} \pa^\al B\|^2
    +C \sum_{\substack{  |\al'|=N-1  }} \| \{\FI-\FP\} \pa^{\al'} u\|^2.
\end{equation}
This completes the estimate \eqref{time.d.claim} by induction. \qed

\medskip

Now, taking a suitable linear combination of \eqref{dr.he.p02} and
\eqref{dr.he.p01}, we see from \eqref{time.d.claim} that
\eqref{dr.he.p02} can be rewritten crudely as
\begin{equation}\notag
    \frac{d}{dt} \CE_N^{\rm h}(t)+\la \CE_N^{\rm h}(t)\leq
    C\sum_{\substack{1\leq |\al|\leq N\\ \al_0=0}}(\|\pa^\al u(t)\|^2+\|\pa^\al [E(t),B(t)]\|^2).
\end{equation}
Here we noticed that from \eqref{def.dNm} for a small $\kappa  >0$ it holds that
$$
\CD_{N}(t)
-
\kappa \sum_{\substack{  |\al'|=N-1  }} \| \{\FI-\FP\} \pa^{\al'} u(t)\|^2
\ge 0.
$$
Then using  the norm \eqref{brief.norm}, we have equivalently shown
\begin{equation}\label{dr.he.p02.1}
    \frac{d}{dt} \CE_N^{\rm h}(t)+\la \CE_N^{\rm h}(t)\leq C\|\na_x U(t)\|_{\CH^{N-1}}^2.
\end{equation}
As before, it follows from Theorem \ref{thm.ls} and the representation \eqref{decom.U} of $U(t)$ that
\begin{multline}
\label{dr.he.p03}
 \|\na_x U(t)\|_{\CH^{N-1}}^2 \leq  C (1+t)^{-\frac{5}{2}} \|U_0\|_{\CH^{N+3}\cap \CZ_1}^2\\
 +C \int_0^t (1+t-s)^{-\frac{5}{2}} \|\nu^{-1/2}\sourceG_1(s)\|^2_{L^2_\xi(H^{N+3}_x)\cap Z_1}ds\\
+ C \left[\int_0^t (1+t-s)^{-\frac{5}{4}} \|\sourceG_2(s)\|_{L^2_\xi(H^{N+3}_x)\cap Z_1}ds\right]^2.
\end{multline}
Recall the estimates from Lemma \ref{lem.h1h2} so that one has
\begin{eqnarray}
&& \|\nu^{-1/2}\sourceG_1(t)\|^2_{L^2_\xi(H^{N+3}_x)\cap Z_1}\leq
C\CE_{N+3,1}(t)\CE_{N+3}(t),\label{g1.high}\\
&&\|\sourceG_2(t)\|_{L^2_\xi(H^{N+3}_x)\cap Z_1}\leq
C\CE_{N+3}(t),\label{g2.high}
\end{eqnarray}
for any $t\geq 0$. Since we suppose that $\eps_{ {N+5},1}$ is
sufficiently small, the first part of Theorem \ref{thm.ns} implies
\begin{equation}  \notag
    \CE_{N+3,1}(t)\leq C \eps_{N+5,1}(1+t)^{-\frac{3}{2}}.
\end{equation}
Putting \eqref{g1.high}-\eqref{g2.high} together the above estimate
into \eqref{dr.he.p03} then gives
\begin{equation}  \notag
 \|\na_x U(t)\|_{\CH^{N-1}}^2 \leq  C  \eps_{N+5,1}(1+t)^{-\frac{5}{2}}.
\end{equation}
Using \eqref{dr.he.p02.1}, this implies from the Gronwall inequality
that \eqref{thm.ns.3} holds true for any $t\geq 0$. This completes
the proof of Theorem \ref{thm.ns}. \qed

\begin{remark}
It is easy to see that, using the similar method as for the proof \eqref{thm.energy.1}, \eqref{thm.energy.2} without velocity weights in the high-order energy functional $\CE_N^{\rm h}(t)$, the Lyapunov inequality can be refined as
$$
\frac{d}{dt}\CE_{N,m}^{\rm h}(t)+\la \CD_{N,m}(t)\leq C(\|\nu^{1/2}\na_x\FP u(t)\|^2+\|\na_x\times B(t)\|^2),
$$
for any $t\geq 0$ provided that $\CE_{N,m-1}(0)$ is sufficiently small, where the integer $m\geq 0$ is given. Thus, starting from the above inequality as for \eqref{dr.he.p01}, it can be proved that
$$
    \CE_{N,m}^{\rm h}(t)\leq C\eps_{N+5,1} (1+t)^{-\frac{5}{2}},
$$
holds for $t\geq 0$ if  $\CE_{N,m-1}(0)$ and $\eps_{N+5,1}$ are sufficiently small.
\end{remark}


\subsection{Decay rates of solutions in  {$L^r_x$}}\label{sec.decayNL.3} In this subsection, we shall prove Corollary \ref{cor}. It suffices to obtain time rates in \eqref{cor.1} and \eqref{cor.2}. They hold true when $r=2$ due to  assumptions of Corollary \ref{cor}, Theorem \ref{thm.ns} and definitions \eqref{def.eNm}, \eqref{def.eNm.h} of $\CE_{N,m}(t)$ and $\CE_{N,m}^{\rm h}(t)$, where
 {
\eqref{app.p2}
}
is used to estimate $\|E(t)\|$.

Then, the rest is to verify
\begin{equation}\label{dr.inf.0}
  \|U(t)\|_{Z_\infty}= \|u(t)\|_{L^2_\xi(L^\infty_x)}+\|[E(t),B(t)]\|_{L^\infty_x}\leq C (1+t)^{-\frac{3}{2}}.
\end{equation}

\begin{lemma}\label{lem.semi.inf}
Let $U_0=[u_0,E_0,B_0]$ satisfy \eqref{comp.con}. Then, it holds that
\begin{equation}\label{lem.semi.inf.1}
    \|\semiG (t)U_0\|_{\CZ_\infty}\leq C(1+t)^{-\frac{3}{2}}(\|U_0\|_{\CZ_1}+\|\na_x^3U_0\|_{\CH^3}),
\end{equation}
for any $t\geq 0$.
\end{lemma}

\begin{proof}
As in \eqref{def.lu1}, let us set $U^I(t)=\semiG (t)U_0$ with
$U^I=[u^I,E^I,B^I]$. From the Sobolev inequality \cite[Proposition
3.8]{Tay}, one has
\begin{equation}\label{lem.semi.inf.p1}
  \|u^I\|_{L^2_\xi(L^\infty_x)} \leq  C \|\na_x u^I\|^{1/2}\|\na_x^2 u^I\|^{1/2}.
\end{equation}
Notice that the proof of this inequality also follows easily from optimizing \eqref{sobolev.trick}.
For the first-order derivative, using \eqref{thm.ls.1} with $m=1, r=1, \ell=5/2$, it follows that
\begin{equation*}
    \|\na_x u^I(t)\|\leq C (1+t)^{-\frac{5}{4}} (\|U_0\|_{\CZ_1}+\|\na_x^3U_0\|_{\CH^1}).
\end{equation*}
For the second-order derivative, similarly letting  $m=2, r=1, \ell=7/2$ in \eqref{thm.ls.1}, one has
\begin{equation*}
    \|\na_x^2 u^I(t)\|\leq C (1+t)^{-\frac{7}{4}} (\|U_0\|_{\CZ_1}+\|\na_x^5U_0\|_{\CH^1}).
\end{equation*}
Putting the above two estimates into \eqref{lem.semi.inf.p1} implies
\begin{equation*}
    \|u^I(t)\|_{Z_\infty}\leq C(1+t)^{-\frac{3}{2}}(\|U_0\|_{\CZ_1}+\|\na_x^3U_0\|_{\CH^3}).
\end{equation*}
Similarly, it holds that
\begin{equation*}
    \|[E^I(t),B^I(t)]\|_{Z_\infty}\leq C(1+t)^{-\frac{3}{2}}(\|U_0\|_{\CZ_1}+\|\na_x^3U_0\|_{\CH^3}).
\end{equation*}
Therefore, \eqref{lem.semi.inf.1} follows.
\end{proof}

\noindent Now, by applying Lemma \ref{lem.semi.inf} to the representation \eqref{decom.U} of $U(t)$, one has
\begin{eqnarray}\label{dr.inf.1}
  \|U(t)\|_{\CZ_\infty} &\leq &  C(1+t)^{-\frac{3}{2}}\|U_0\|_{\CH^{6}\cap\CZ_1}\\
  &&+C\int_0^t(1+t-s)^{-\frac{3}{2}}\|g(s)\|_{L^2_\xi(H^{6}_x)\cap Z_1}ds.
  \notag
\end{eqnarray}
Notice from Lemma \ref{lem.h1h2} that
\begin{equation*}
    \|g(t)\|_{L^2_\xi(H^{6}_x)\cap Z_1}\leq C\CE_{6,2}(t).
\end{equation*}
Now since $\eps_{8,2}$ in \eqref{def.id.jm} is sufficiently small, it follows from \eqref{thm.ns.2} that
\begin{equation*}
    \CE_{6,2}(t)\leq C\eps_{8,2}(1+t)^{-\frac{3}{2}}.
\end{equation*}
Putting the above estimates into \eqref{dr.inf.1} leads to  \eqref{dr.inf.0}. This completes the proof of Corollary \ref{cor}. \qed








\subsection*{Acknowledgments}
The authors thank the anonymous referee for many helpful and valuable comments that substantially improve the presentation of this paper. R.-J. Duan was partially supported by the Direct Grant 2010/2011, and would acknowledge RICAM, Austrian Academy of Sciences when this project was started there in 2009, and also thank Professor Shuichi Kawashima for sending him some recent work \cite{IHK,IK,UWK}. R.M. Strain was partially supported by the NSF grant DMS-0901463 and an Alfred P. Sloan Foundation Research Fellowship.


\frenchspacing
\bibliographystyle{plain}





\end{document}